\documentclass[article,11pt]{amsart}
\usepackage[top=25mm,bottom=25mm,left=25mm,right=25mm]{geometry}
\usepackage{graphicx}

\usepackage{color}

\newtheorem{lemma}{Lemma}[section]
\newtheorem{proposition}{Proposition}[section]
\newtheorem{theorem}{Theorem}[section]
\newtheorem{corollary}{Corollary}[section]
\newtheorem{definition}{Definition}[section]
\newtheorem{example}{Example}[section]
\newtheorem{remark}{Remark}[section]
\newtheorem{assumption}{Assumption}[section]

\makeatletter
\def\section{\@startsection{section}{1}%
\z@{1\linespacing\@plus\linespacing}{1\linespacing}%
{\bf\centering}}
\def\subsection{\@startsection{subsection}{0}%
\z@{\linespacing\@plus\linespacing}{\linespacing}%
{\bf}}
\makeatother

\makeatletter
\@addtoreset{equation}{section}
\makeatother

\DeclareMathOperator{\loc}{loc}
\DeclareMathOperator{\tr}{Tr}

\newcommand{\cO}{\mathcal{O}}

\newcommand{\cG}{\mathcal{G}}

\newcommand{\cK}{\mathcal{K}}

\newcommand{\cB}{\mathcal{B}}

\newcommand{\cV}{\mathcal{V}}

\newcommand{\cM}{\mathcal{M}}

\newcommand{\R}{\mathbb{R}}
\newcommand{\Z}{\mathbb{Z_+}}
\newcommand{\1}{\mathbf{1}}
\newcommand{\pr}{\mathbf{P}}
\newcommand{\qpr}{\mathbb{Q}}

\newcommand{\ex}{\mathbf{E}}

\newcommand{\N}{\mathbb{N}}

\begin{document}
\title[Integrated density of states on the Sierpi\'nski Gasket]
{Integrated density of states for Poisson-Schr\"odinger perturbations of subordinate Brownian motions on the Sierpi\'nski Gasket}
\author{Kamil Kaleta and Katarzyna Pietruska-Pa{\l}uba}

\address{K. Kaleta \\ Institute of Mathematics, University of Warsaw, ul. Banacha 2, 02-097 Warszawa and Institute of Mathematics and Computer Sciences, Wroc{\l}aw University of Technology, Wyb. Wyspia\'nskiego 27, 50-370 Wroc{\l}aw, Poland}
\email{kkaleta@mimuw.edu.pl, kamil.kaleta@pwr.wroc.pl}

\address{K. Pietruska-Pa{\l}uba \\ Institute of Mathematics \\ University of Warsaw
\\ ul. Banacha 2, 02-097 Warszawa, Poland}
\email{kpp@mimuw.edu.pl}

\begin{abstract}
{We prove the existence of the integrated density of states for subordinate
Brownian motions in presence of the Poissonian random potentials on the Sierpi\'nski gasket. }


\bigskip
\noindent
2010 {\it MS Classification}: {Primary 60J75, 60H25, 60J35; Secondary 47D08, 28A80}
\end{abstract}

\footnotetext{
Research supported by the National Science Center (Poland) internship
grant on the basis of the decision No. DEC-2012/04/S/ST1/00093 and by the Foundation for Polish Science.}

\maketitle

\baselineskip 0.5 cm

\bigskip\bigskip
\section{Introduction}

The integrated density of states is one of the most important object in large-scale quantum mechanics.
In random physical models with unbounded state-space it is usually difficult to describe possible energy levels of the system (i.e. the
eigenvalues of the Hamiltonian). Such a situation arises e.g. when
the Hamiltonian is a random Schr\"{o}dinger operator
\[H=-\Delta+V,\]
where $\Delta$ is the usual Laplacian in $\mathbb R^d$  and $V=V(x,\omega)\geq 0$  is a
sufficiently regular random field. The spectrum of such an operator is typically  not discrete and therefore hard to investigate,
but some of its properties are captured by the properties of the integrated density of states (IDS) of the system (see \cite[Chapter VI]{bib:Car-Lac}).

To define this object, one considers the operator $H$ constrained to a smooth bounded region $\Omega$ (a box, for example), with either Dirichlet or Neumann boundary conditions on $\partial\Omega.$ This operator, $H^\Omega,$ gives rise to a Hilbert-Schmidt semigroup of operators, $P_t^\Omega={\rm e}^{-tH^\Omega},$ and the spectrum of $H^\Omega$ is discrete. Its eigenvalues can be ordered:
\[0\leq\lambda_1^\Omega\leq \lambda_2^\Omega\leq...\to\infty\]
One then builds random empirical measures based on these spectra and normalizes them by dividing by the volume of $\Omega$:
\[\ell^\Omega(\cdot)\stackrel{def}{=}\frac{1}{|\Omega|}\sum_{k=1}^\infty \delta_{\lambda_k^\Omega}(\cdot).\]
If these measures have a vague limit when $\Omega\nearrow\mathbb R^d,$  then this limit is called the integrated density of states for the system.

When $V$ exhibits some additional ergodic properties, then the limit $\ell$ is a nonrandom Radon measure on $\mathbb R_+:=[0,\infty).$ Its properties near zero are of special interest --
in many cases, one sees the so-called Lifschitz singularity:
the quantity $\ell[0,\lambda)$ decays faster, when $\lambda\to 0^+,$ than its
counterpart with no potential: the decay rate is roughly $\exp{[-(const/\lambda^\kappa)]}$, with a positive parameter $\kappa.$  This behaviour was first discovered by Lifschitz \cite{bib:Lif} on physical grounds and proven rigorously in \cite{bib:Nakao,bib:Pastur}.
 We also refer to
\cite{bib:Szn1} for an alternate proof of the Lifschitz singularity in presence of the killing obstacles. This situation can be understood as a limiting case of the interaction with  potential.

Nonrandom IDS arise  e.g. in the particular case of Poisson random fields
\[V(x,\omega)=\int_{\mathbb{R}^d}W(x-y)\,\mu^\omega({\rm d}y),\]
where $\mu^\omega$ is the counting measure {on $\mathbb R^d$} of a realization of a Poisson point process {over $(\Omega, \cM, \qpr)$} and $W$ is a sufficiently regular profile function. In this case, one can just write down the formula for the Laplace transform of the IDS:
\begin{align} \label{eq:Lt}
L(t)=\frac{1}{(2\pi t)^{d/2}}\mathbf E^t_{0,0}\mathbb{E_Q}{\rm e}^{-\int_0^t V(B_s, \omega)\,{\rm d}s},
\end{align}
where $B_s$ is the Brownian motion on $\mathbb R^d$, {and $\mathbf E^t_{x,x}$ and $\mathbb {E_Q}$ are expectations with respect to the corresponding Brownian bridge measure and the probability measure $\qpr$, respectively. The formula \eqref{eq:Lt} is a direct consequence of the Feynman-Kac formula, stationarity of the potential $V$ and the translation invariance of the Brownian motion. }

A remarkable feature of the limit is that it is the same for both Dirichlet and Neumann boundary conditions. For spaces other than the Euclidean space, this is not always the case:
for example, in the hyperbolic space, these two limits are distinct \cite{bib:Szn-hyp1,bib:Szn-hyp2}.

For more information on {IDS} in the {classical} (i.e. the Brownian motion) case we refer e.g. to  the book \cite{bib:Car-Lac}.

\smallskip

Similar existence result and similar representation formula {for the Laplace transform of IDS} can be also derived for generalized Schr\"odinger operators of the form
\[H=-\mathcal L +V,\]
where $\mathcal L$ is the generator of a symmetric jump L\'{e}vy process and $V$ is a sufficiently regular Poissonian potential. This was done in  \cite{bib:Okura}
by extending the method of \cite{bib:Nakao}. {A very important example to this class of operators are fractional Schr\"{o}dinger operators $(-\Delta)^{\alpha/2} + V$, $\alpha \in (0,2)$, and relativistic Schr\"{o}dinger operators $(-\Delta+\mu^{2/\alpha})^{\alpha/2} - \mu + V$, $\alpha \in (0,2)$, $\mu >0$, which correspond respectively to the jump rotation invariant $\alpha$-stable and relativistic $\alpha$-stable processes perturbed by the random potential $V$ in $\mathbb R^d$. }

{The theory for generalized Schr\"odinger operators underwent rapid development, stimulated by problems of relativistic quantum physics, at the end of the 20th century. There is a wide literature concerning the spectral and analytic properties of nonlocal Schr\"odinger operators (see, e.g., \cite{bib:CMS, bib:Z, bib:KL} and references therein). Most of it has been strongly influenced by Lieb's investigations on the stability of (relativistic) matter \cite{bib:LS}. }

{Random perturbations of stochastic processes} in irregular spaces, such as fractals, have been considered as well. The Laplacian should be replaced there by the generator of the Brownian motion: by the Brownian motion one understands a Feller diffusion which remains invariant under local symmetries of the state-space. Such a process has been constructed on nested fractals or on the Sierpi\'{n}ski carpet
\cite{bib:BBa,bib:BP, bib:Kum, bib:Lin} and proven to be unique \cite{bib:BBaKT,bib:Sab}.
The existence of IDS on the Sierpi\'{n}ski gasket with killing Poissonian obstacles, and its Lifschitz singularity have been established in \cite{bib:KPP-PTRF}. The proof of existence from this paper {directly extends} to other nested fractals. {Later, the existence and the Lifschitz singularity of Brownian IDS for Poissonian obstacles and Poissonian potentials with profiles of finite range on nested fractals was also proven in \cite{bib:Sh}. The argument in that paper is based on the locality and scaling properties of the corresponding Dirichlet forms and cannot be adapted to the nonlocal case (i.e. for jump Markov processes) and Poissonian potentials with profiles of infinite range.}

\smallskip

Subordinate Brownian motion on fractals can be defined as well, by means of the classical subordination method \cite{bib:SSV}. However, it is usually difficult to establish properties of such processes: lack of translation invariance and lack of continuous scaling for the Brownian motion on fractals come as a {main difficulty here}. Properties of the subordinate $\alpha-$stable processes on $d-$sets, including nested fractals, have been investigated in \cite{bib:BSS} (cf. \cite{bib:CKu}). The boundary Harnack principle for functions harmonic with respect to such processes in natural cells of Sierpi\'nski gasket and carpet was studied in \cite{bib:BSS, bib:St1} and for arbitrary open subsets of Sierpi\'nski gasket in \cite{bib:KKw}. Very recently, in \cite{bib:BKK}, it was established for more general Markov processes on measure metric spaces, including some subordinate processes on simple nested fractals and Sierpi\'nski carpets. Basic spectral properties for subordinate processes on measure metric spaces, including a wide range of fractals, were studied in \cite{bib:CS}.

\smallskip

In the present paper, we prove the existence of the integrated density of states for the subordinate Brownian motions {in presence of the random Poissonian potential} on the Sierpi\'{n}ski gasket. Again, lack of {the homogeneity of the state space and lack of} the traslation invariance for the Brownian motion
(and, consequently, for the subordinate Brownian motions) form {a major}
obstacle here.

To establish the existence of the IDS, we investigate the Laplace transforms of empirical measures $\ell^\Omega(\cdot)$ arising from the problem for {Feynman-Kac semigroups of both the killed and the 'reflected' processes} in 'big boxes' and then we follow the scenario, previously used e.g. in \cite{bib:Szn-hyp1, bib:KPP-PTRF}
for the Brownian motion with killing obstacles:
\begin{enumerate}
\item[(1)] to prove that the averages of the Laplace transforms, with respect to the Possonian measure, do converge,
\item[(2)] to prove that their variances converge fast enough to permit a Borel-Cantelli lemma argument to get the desired convergence.
\end{enumerate}

While for the Brownian motion {killed by the Poissonian obstacles} part (1) was easy and proof of part (2) was longer, now this is part (1) which gets harder {and its proof} is the major step in obtaining the existence of IDS {(recall that also in the case of L\'evy processes in the Poissonian potentials in Euclidean spaces, the convergence described in (1) is a direct consequence of the homogeneity of the space and the translation invariance of the process). We have to take into account specific geometric properties of the Sierpi\'{n}ski gasket, and, therefore, we propose a new regularity condition on the (two argument{,  possibly with infinite range}) profile functions $W$, under which the Poissonian potential $V$ has the desired stationarity property. It seems to be natural for the gasket. An essential feature of our method is that in fact we prove the convergence as in (1) for a 'periodization' of the Poissonian potential $V$ instead of its initial shape.} Even in the Brownian case the potentials with profiles of infinite range on the Sierpi\'nski gasket were not studied so far.

Along the way, we also get that the limit is the same for both the Dirichlet and the `Neumann' {approach}. We use the term `Neumann' in analogy to the Brownian motion case: the process we are using in  this case is a counterpart of the `reflected Brownian motion' on the gasket from \cite{bib:KPP-PTRF}, and in the Euclidean case, the reflected Brownian motion has the Neumann Laplacian as its generator.

In the forthcoming paper \cite{bib:KaPP} we study the Lifschitz singularity of IDS for a class of subordinate Brownian motions subject to Poisson interaction on the Sierpi\'nski gasket.
Our proofs, both in the present and in the forthcoming paper,
hinge on the construction of the `reflected' subordinate Brownian
motions, which in turn rely on the exact labeling of the vertices
on the gasket. It would be  interesting to establish the existence and other properties of the IDS for such a problem in fractals more general than the Sierpi\'{n}ski gasket, in
particular on those fractals on which such labeling does not work.

The paper is organized as follows.
In Section \ref{sec:basic} we collect essentials on the constructions Sierpi\'{n}ski gasket, properties of Brownian motion and subordinate processes on the gasket. We also construct the
`reflected subordinate process' and prove its basic properties. Then we recall basic facts concerning Feynman-Kac semigroups, with both deterministic and random potentials. In Section
\ref{sec:conv} we prove the main result of this article -- the existence of the IDS for
the subordinate Brownian motion on the Sierpi\'{n}ski gasket influenced by a Poissonian
potential (Theorem \ref{th:main-fractal}). Along the way we establish that the IDS for the
Dirichlet and the Neumann {approach} coincide. In Section \ref{sec:ex} we conclude the paper with examples
of admissible profile functions, with both finite and infinite range.

\section{Basic definitions and preliminary results}\label{sec:basic}

\subsection{Sierpi\'nski Gasket}\label{sec:gasket}
The infinite Sierpi\'nski Gasket we will be working with is
defined as a blowup of the unit gasket, which in turn is defined
as the fixed point of the iterated function system
in $\mathbb R^2,$ consisting of three maps:
\[
\phantom{}\ \ \ \ \ \ \phi_1(x)=\frac{x}{2},\ \ \ \ \   \phi_2(x)=
\frac{x}{2}+\left(\frac{1}{2},0\right) \ \ \ \ \ \
\phi_3(x)=\frac{x}{2}+\left(\frac{1}{2}, \frac{\sqrt 3}{2}\right). \hfill \]
The unit gasket, $\mathcal G_0,$ is the unique compact subset of
$\mathbb{R}^2$ such that
\[\mathcal G_0= \phi_1(\mathcal G_0)\cup \phi_2(\mathcal G_0)\cup\phi_3(\mathcal G_0).\]
Denote by $\mathcal V^{(0)}=\{a_1,a_2,a_3\}=\{(0,0), (1,0),
(\frac{1}{2},\frac{\sqrt 3}{2})\}$ the set of its vertices. Then
we set:
\[\mathcal G_n=2^n \mathcal G_0 =((\phi_1^{-1}))^n(\mathcal G_0),\]
\[\mathcal G=\bigcup_{n=1}^\infty \mathcal G_n,\]
and inductively, for $n=1,2,...$:
\[\mathcal V^{(n+1)}= \mathcal V^{(n)}\cup \{2^na_2+\mathcal V^{(n)}\} \cup\{2^na_3 +\mathcal V^{(n)}\}, \]
\[\mathcal V_0=\bigcup_{n=0}^\infty \mathcal V^{(n)},\;\;\; \mathcal V_{M+1}= 2^M \mathcal V_0.\]
Elements of $\mathcal V_M$ are exactly the vertices of all triangles of
size $2^M$ that build up the infinite gasket.

We equip the gasket with the shortest path distance
$d(\cdot,\cdot)$: for $x,y\in \mathcal V_0,$ $d(x,y)$ is the infimum of
Euclidean lengths of all paths, joining $x$ and $y$ in the gasket.
For general $x,y\in \cG,$  $d(x,y)$ is obtained by a limit
procedure. This metric is equivalent to the usual Euclidean metric
inherited from the plane,
\[|x-y|\leq  d(x,y)\leq 2|x-y|.\]
    Observe that
$\mathcal G_M= B(0, 2^M),$ where the ball is taken in either the
Euclidean or the shortest path metric.

By $m$ we denote the Hausdorff measure on $\mathcal G$ in
dimension $d_f=\frac{\log 3}{\log 2}.$ It is normalized to have
$m(\mathcal G_0)=1.$ The number $d_f,$ being the Hausdorff dimension of the gasket $\mathcal G,$   is  sometimes  called its fractal
dimension. Another characteristic number of
$\mathcal G,$ namely $d_w=\frac{\log 5}{\log 2}$ is called the
walk dimension of $\mathcal G.$ The spectral dimension of
$\mathcal G$ is $d_s=\frac{2d_f}{d_w}.$

We will need the following estimate.

\begin{lemma}\label{lem:kolnierzyk}
Let $M\in\mathbb Z$ and $0<r<2^M.$
Then there exists a constant $c\geq 1$  ($c=1$ if $r$ is binary) such that
\begin{equation}\label{eq:kolnierzyk}
\frac{1}{c}2^Mr^{d_f-1}\leq m \left(B(0,2^M) \setminus B(0,2^M-r)\right) \leq c2^Mr^{d_f-1},
\end{equation}
where the balls are  taken in  the geodesic metric.
\end{lemma}
\begin{proof}
When $r$ is binary, i.e.  $r=2^k,$ $k\in \mathbb{Z},$
then the set $B(0,2^M) \setminus B(0,2^M-r)$ consists of $2^M/r$ triangles of size $r=2^k$
each, and so $m\left(B(0,2^M) \setminus B(0,2^M-r)\right)= \frac{2^M}{r}\cdot r^{d_f}=2^Mr^{d_f-1}.$

When $r$ is not binary, then let $r_0$ be the biggest binary number not exceeding $r,$
i.e. $r_0=2^k$, with $k$ determined uniquely by $2^k \leq r<2^{k+1}.$
As $r_0\leq r< 2 r_0$ and
\[B(0,2^M) \setminus B(0,2^M-r_0) \subseteq B(0,2^M) \setminus B(0,2^M-r) \subseteq B(0,2^M) \setminus B(0,(2^M-2r_0)_+) ,\]
the statement follows from the above.
\end{proof}

In the sequel, we will use a projection from $\mathcal G$ onto
$\mathcal G_M,$ $M=0,1,2,...$ To define it properly, we first put
labels on the set $\mathcal V_0$ (see \cite{bib:KPP-PTRF}).

Observe that  {$\mathcal
V_0\subset(\mathbb{Z}_+)a_2+(\mathbb{Z}_+)a_3$ (recall that  $a_2=(1,0)$
and $a_3=\left(\frac{1}{2},\frac{\sqrt 3}{2}\right))$.} Next,
consider the commutative three-element group $\mathbb{A}_3$ of even
permutations of 3 elements,  $\{a,b,c\},$ i.e. $\mathbb{A}_3=\{id,
 {(a\mapsto b\mapsto c), (a\mapsto c\mapsto b)} \},$ and we denote  {$p_1=(a\mapsto b\mapsto c),$ $p_2=(a\mapsto c\mapsto b)$.}
The mapping
\[\mathcal V_0\ni x=na_2+ma_3\mapsto p_1^n\circ p_2^m\in\mathbb{A}_3\]
is well defined, and for $x\in \mathcal V_0$ we put $l(x)= (p_1^n\circ
p_2^m)(a).$  {The vertices of any triangle of size 1 are of the
form $\{na_2+ma_3, (n+1)a_2+ma_3, na_2+(m+1)a_3\} $ with certain $m,n$, and so its labels are $p_1^n\circ p_2^m(a)=:l_0, $
$p_1^{n+1}\circ p^m(a)=p_1(l_0),$ $p_1^n\circ p_2^{m+1}(a)=p_2(l_0).$ We check by inspection that for any $l_0\in\{a,b,c\}$ these three labels are distinct. Consequently,} every triangle of size 1 has its vertices
labeled `$a,b,c$'  {(see Fig. 1)}. Note that this property extends to every
triangle of size $2^M,$ which corresponds to putting labels on the
elements of $\mathcal V_M:$ every triangle of size $2^M$ has three distinct labels on its vertices.  {This is so because the vertices of any gasket triangle of size $2^M$ are of the form
$n2^Ma_2+m 2^Na_3, (n+1) 2^{M}a_2+m2^Ma_3, n2^M a_2+(m+1)2^Na_3,$ and so its labels are $\{l:= p_1^{n\cdot2^M}\circ p_2^{m\cdot2^M}(a),$ $p_1^{2^M}(l),$ $p_2^{2^M}(l)\}.$ As $p_1^{2^M}=p_2, p_2^{2^M}=p_1$ for
$M$ odd  and $p_1^{2^M}=p_1, p_2^{2^M}=p_2$ for $M$ even, in either case the vertices of such a triangle have  three distinct labels assigned: $l, p_1(l), p_2(l).$}

\begin{figure}[t!] \label{fig:fig1}
\centering
\includegraphics[width=13cm]{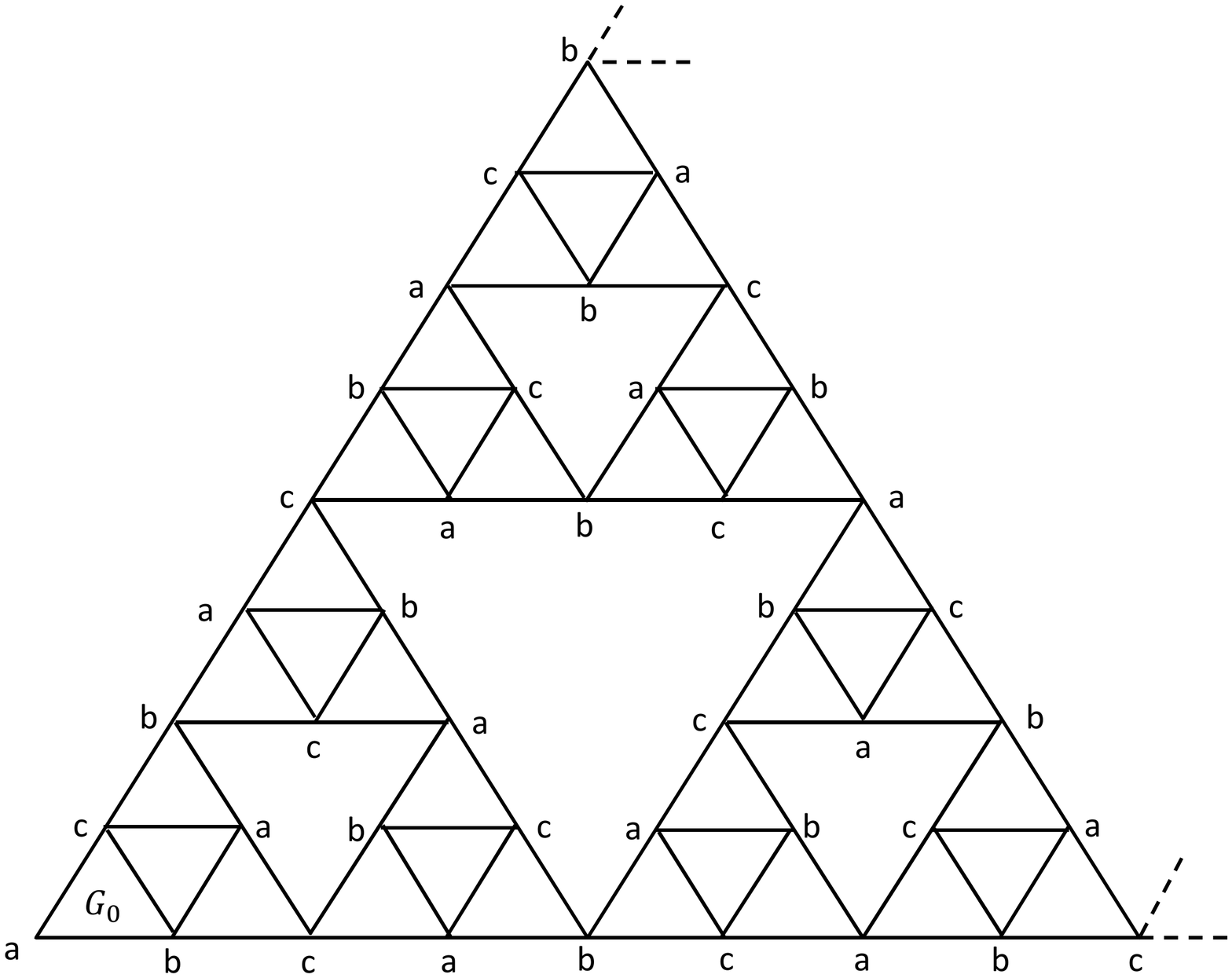}
\caption{Labels on $\mathcal V_0.$}
\end{figure}

\begin{figure}[t!] \label{fig:fig2}
\centering
\includegraphics[width=13cm]{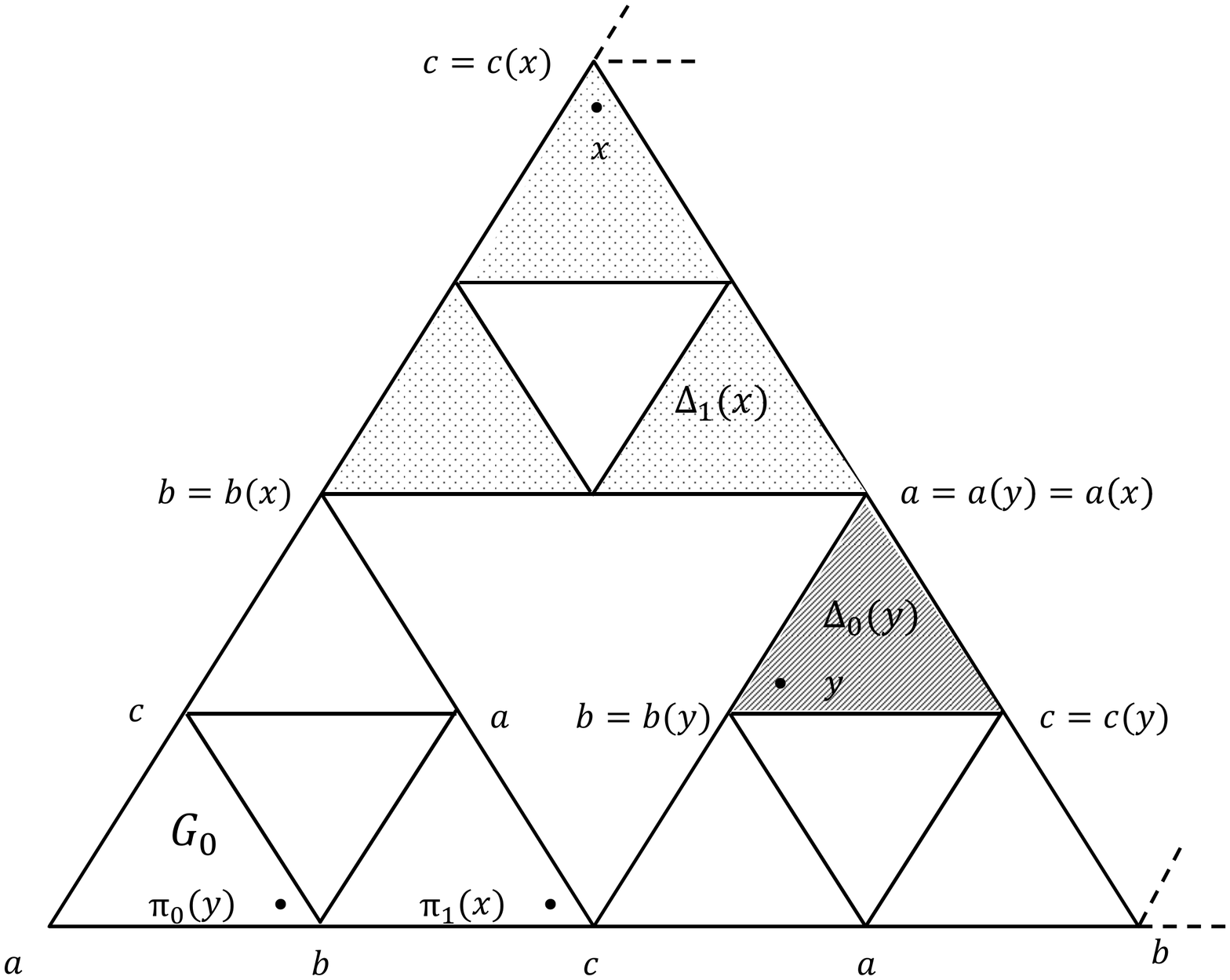}
\caption{Illustration of the action of $\pi_0$ and $\pi_1$.}
\end{figure}

Let $M\geq 0 $ be fixed. For $x\in \mathcal G\setminus \mathcal V_M,$ there
is a unique triangle of size $2^M$ that contains $x,$
$\Delta_M(x),$ and so $x$ can be written as $x=x_a a(x)+x_b
b(x)+x_c c(x),$ where $a(x),$ $b(x),$ $c(x)$ are the vertices of
$\Delta_M(x)$ with labels $a,b,c$  and $x_a,x_b,x_c\in (0,1),$
$x_a+x_b+x_c=1.$ Then we define the projection:
\[\mathcal G\setminus \mathcal V_M\ni x\mapsto \pi_M(x)= x_a\cdot a(M)+x_b\cdot b(M) +x_c\cdot c(M)\in \mathcal G_M, \]
where $a(M),$ $b(M),$
 $c(M)$ are the vertices of the triangle $\mathcal G_M$ with
 corresponding labels $a,b,c$  {(see Fig. 2).}
 When $x\in \mathcal V_M,$ then $x$ itself has a label assigned and it
 then mapped onto the corresponding vertex of $\mathcal G_M.$

For every $M=0,1,2,...,$ the gasket $\mathcal G_{M+1}$ consists of
$3$ isometric copies of $\mathcal G_M.$ These copies are denoted
by $\mathcal G_M^{(i)},$ $i=1,2,3.$  {For $i \neq j$ one has $m\left(\mathcal G_M^{(i)} \cap
\mathcal G_M^{(j)}\right)=0$}. Also, denote by
$\pi_{M,i}$ the restriction of $\pi_{M}$ to $\mathcal G^{(i)}_M$.

\

\subsection{The Brownian motion and subordinate processes on gaskets}

\subsubsection{Brownian motion}
The Brownian motion on the two-sided infinite gasket (by the two-sided gasket we mean
the set $\mathcal G^*=\mathcal G\cup i(\mathcal G),$ where $i$ is the reflection
of $\mathbb R^2$  with respect to the $y-$axis) was first defined in \cite{bib:BP}. It is a strong Markov and Feller process $\tilde Z=(\tilde Z_t, \pr_x)_{t\geq 0,x\in\mathcal G^*}$, whose transition density with respect to the Hausdorff measure is symmetric in its space variables, continuous, and fulfils
 the following subgaussian estimates:
\begin{eqnarray}
\label{eq:gaussian}
c_1 t^{-d_s/2} {\rm e}^{-c_2 \left(\frac{d(x,y)}{t^{1/d_w}}\right)^{d_w/(d_w-1)}} \leq \tilde g(t,x,y) \leq c_3 t^{-d_s/2} {\rm e}^{-c_4 \left(\frac{d(x,y)}{t^{1/d_w}}\right)^{d_w/(d_w-1)}}\hskip -1cm,\\
 \nonumber \phantom{} \hskip 7cm \quad x,y \in \mathcal G^*, \quad t>0,
\end{eqnarray}
with positive constants $c_1,...,c_4.$
It is not hard to see that  the Brownian motion $Z$ on the one-sided Sierpi\'{n}ski gasket $\mathcal G$ obtained from $\tilde Z$ by the projection $\mathcal G^*\to \mathcal G$,
whose transition density is equal to
\(g(t,x,y)= \tilde g(t,x,y)+\tilde g(t,x,i(y))\) for $x\neq 0$ and twice this quantity when $x=0,$ shares all these properties, the subgaussian estimates included (with possibly worse constants $c_i$). We stick to the estimate (\ref{eq:gaussian}) for $g(t,x,y)$
as well.

\subsubsection{Subordinate Brownian motion}\label{sec:sbm}
Let $S=(S_t, \pr)_{t \geq 0}$ be a subordinator, i.e. an increasing L\'{e}vy process taking values in $[0,\infty]$ with $S_0=0$.
The law of $S$, which will be denoted by $\eta_t({\rm d}u)$, is determined by the Laplace transform
$\int_0^{\infty} {\rm e}^{-\lambda s} \eta_t({\rm d}s) ={\rm e}^{-t\phi(\lambda)},$ $\lambda>0.$
The function $\phi:(0,\infty) \to [0,\infty)$ is called the Laplace exponent of $S$ and it has the representation
\begin{align}
\label{def:phi}
\phi(\lambda) = b \lambda + \int_0^{\infty} (1-{\rm e}^{-\lambda s})\nu({\rm d} s),
\end{align}
where $b\geq 0$ is called the drift term and $\nu$, called the L\'evy measure of $S$, is a $\sigma$-finite measure on $(0,\infty)$ satisfying $\int_0^{\infty} (s \wedge 1) \nu({\rm d}s) < \infty$. It is well known that when a function $\phi:(0,\infty) \to \R$ satisfies $\lim_{\lambda \to 0^{+}} \phi(\lambda)=0$ then it has can be represented by \eqref{def:phi} if and only if it is a Bernstein function \cite{bib:SSV}. If the measure $\eta_t({\rm d}u)$ is absolutely continuous with respect to the Lebesgue measure, then the corresponding density is  denoted by $\eta_t(u)$. For more properties of subordinators and Bernstein functions we refer to \cite{bib:Ber, bib:Ber2, bib:SSV}.

We always assume that $Z$ and $S$ are independent. The process $X=(X_t, \pr_x)_{t\geq 0,x\in\mathcal G}$ given by
$$
X_t := Z_{S_t}, \quad t \geq 0,
$$
is called the subordinate Brownian motion on $\cG$ (via subordinator $S$). It is also a symmetric Markov process {with respect to its natural filtration (assumed to fulfil the usual conditions)}, with c\`adl\`ag paths. Its transition probabilities are given by
$$
p(t,x,A) = \int_0^{\infty} \int_A g(u,x,y) m({\rm d}y) \eta_t({\rm d}u), \quad t>0, \quad x \in \cG, \quad A \in \cB(\cG).
$$
Throughout the paper we impose some assumptions on the subordinator $S$ which provide  sufficient regularity of the process $X$.

\begin{assumption} \label{ass:ass1} For every $t>0$ the following holds.
\begin{itemize}
\item[]
\begin{equation}\label{eq:assum1}
\eta_t(\left\{0\right\})=0 \quad \text{and} \quad \int_{0^{+}}^\infty \frac{1}{u^{d_s/2}}\,\mathbb \eta_t({\rm d}u)=:c_0(t) <\infty.
\end{equation}
\item[]
\begin{equation}\label{eq:log}
{\int_1^{\infty} \eta_t(u,\infty) \frac{{\rm d} u}{u}<\infty.}
\end{equation}
\end{itemize}
\end{assumption}

It is clear that when $S_t=t$ and $X_t=Z_t$ (in this case, $\eta_t({\rm d}u)=\delta_t({\rm d}u)$ and $\phi(\lambda)=\lambda$), then both conditions \eqref{eq:assum1} and \eqref{eq:log} are satisfied.

\begin{remark} \label{rem:rem1}
\noindent
{\rm {
\begin{itemize}
\item[(1)] The condition \eqref{eq:log} is tantamount to $\sum_{M=1}^\infty \eta_t(a^M, \infty)<\infty$, for every $a>1$. This summation property will be often explicitly used below. Moreover, one can verify that \eqref{eq:log} holds when $\int_1^{\infty} \log u \ \eta_t({\rm d}u) < \infty$ for all $t>0$.
\item[(2)] In most cases, the measure $\eta_t(\cdot)$ is not explicitly given, but the corresponding Laplace exponent $\phi$ is known. In this case, very often,  the integral condition in \eqref{eq:assum1} can be verified by using Tauberian theorems of exponential type (see, e.g., \cite{bib:F, bib:Kas}). For example, when $\phi(\lambda) \geq c \lambda^{\gamma}$ for all $\lambda > \lambda_0$ with some $c, \lambda_0 >0$ and $\gamma \in (0,1)$, then by \cite[Theorem 2.1 (ii)]{bib:F} for every $t>0$ there is $\tilde c >0$ such that $\eta_t(0,u] \leq e^{- \tilde c u^{-\gamma/(1-\gamma)}}$ for sufficiently small $u>0$, and the integral condition in \eqref{eq:assum1} holds. Furthermore, when $\phi$ is unbounded (in this case $S$ is not a compound Poisson process, see \cite{bib:Ber}), then for every $t>0$ the distribution of $S$ does not charge $\left\{0\right\}$, which is exactly the first part of \eqref{eq:assum1}. Also, in Lemma \ref{lm:verlog} below, we give a simple estimate which may be used in verification of \eqref{eq:log} when $\phi$ is known.
\end{itemize}}
}
\end{remark}
{
\begin{lemma}\label{lm:verlog}
Let $S$ be a subordinator with Laplace exponent $\phi$ and  {law $\eta_t(du)$ such that $\eta_t(\left\{0\right\})=0$ for every $t>0$}. Then we have
$$
\int_1^{\infty} \eta_t(u, \infty)\frac{{\rm d}u}{u} \leq t \int_0^1 \frac{\phi(\lambda)}{\lambda}\, {\rm d} \lambda + e^{-1}, \quad t > 0.
$$
\end{lemma}}
{
\begin{proof}
Fix $t>0$. Direct integration by parts gives for every $\lambda>0$ that
$$
e^{-t\phi(\lambda)} = \int_0^{\infty} e^{-\lambda u} \eta_t({\rm d}u) = \lambda \int_0^{\infty} e^{-\lambda u} \eta_t(0,u]\, {\rm d}u \leq 1 - \lambda \int_1^{\infty} e^{-\lambda u} \eta_t(u,\infty) \rm{d}u,
$$
and, consequently,
$$
\int_1^{\infty} e^{-\lambda u} \eta_t(u,\infty) {\rm d}u \leq \frac{1-e^{-t\phi(\lambda)}}{\lambda} \leq \frac{t\phi(\lambda)}{\lambda}.
$$
By integrating this inequality in $\lambda$ over the interval $(0,1)$ and changing the order of integrals on the left hand side, we finally get
$$
\int_1^{\infty} \eta_t(u,\infty) \frac{{\rm d}u}{u} \leq t \int_0^1 \frac{\phi(\lambda)}{\lambda}\, {\rm d} \lambda + e^{-1},
$$
which completes the proof.
\end{proof}
}

Our Assumption \ref{ass:ass1} is satisfied by a wide class of subordinators. Below we discuss only several examples which are of special interest. For further examples we refer the reader to \cite{bib:BBKRSV, bib:SSV}.

\begin{example} {\rm In some cases, the densities of measures $\eta_t$ exist and precise bounds for them are known. However, for all examples listed here the Laplace exponent $\phi$ is explicitly given and Assumption \ref{ass:ass1} can be verified by using Lemma \ref{lm:verlog} and Tauberian theorems.
\noindent

\begin{itemize}
\item[(1)] \emph{$\alpha/d_w$-stable subordinators.} Let $\phi(\lambda)=\lambda^{\alpha/d_w}$, $\alpha \in (0,d_w]$. It is well known that in this case the measure $\eta_t({\rm d}u)$ is absolutely continuous with respect to Lebesgue measure, the scaling property $\eta_t(u)=t^{-d_w/\alpha} \eta_1(t^{-d_w/\alpha}u)$ holds, and $\eta_1(u) \approx c(\alpha) u^{-1-\alpha/d_w}$, when $u \to \infty$. When $\alpha \in (0,d_w)$, then the subordination via such subordinator leads to the purely jump process $X$ which is called \emph{$\alpha$-stable process} on $\cG$. The case $\alpha=d_w$ is different. As mentioned just after Assumption \ref{ass:ass1}, in this case the process $Z$ remains unchanged. For properties of the subordinate $\alpha$-stable processes we refer to \cite{bib:BSS, bib:KKw} (cf. also \cite{bib:CKu}).
\smallskip
\item[(2)] \emph{Mixture of several purely jump stable subordinators.} Let $\phi(\lambda)=\sum_{i=1}^n \lambda^{\alpha_i/d_w}$, $\alpha_i \in (0,d_w)$, $n \in \N$. Many of basic properties of the process subordinated via this subordinator can be established in a similar way as in \cite{bib:BSS}.
\smallskip
\item[(3)] \emph{$\alpha/d_w$-stable subordinator with drift.} Let $\phi(\lambda)=b\lambda + \lambda^{\alpha/d_w}$, $\alpha \in (0,d_w)$, $b>0$. Then the corresponding subordinator is a sum of a pure drift subordinator $b t$ and the purely jump $\alpha/d_w$-stable subordinator. In this case, $\phi(\lambda) \approx \lambda$ for $\lambda \to 0^{+}$, and $\phi(\lambda) \approx \lambda^{\alpha/d_w}$ for $\lambda \to \infty$.
\item[(4)] \emph{Relativistic $\alpha/d_w$-stable subordinator.} Let $\phi(\lambda)=(\lambda+\mu^{d_w/\alpha})^{\alpha/d_w}-\mu$, $\alpha \in (0,d_w)$, $\mu>0$. The subordination via such a subordinator leads to the so-called \emph{relativistic $\alpha$-stable} process on $\cG$. Since $\phi(\lambda) \approx \lambda$ for $\lambda \to 0^{+}$, and $\phi(\lambda) \approx \lambda^{\alpha/d_w}$ for $\lambda \to \infty$, similarly as before, Assumption \ref{ass:ass1} is satisfied.
\smallskip
\item[(5)] If $S$ is a subordinator with Laplace exponent $\phi(\lambda)=\lambda^{\alpha/d_w}[\log(1+\lambda)]^{\beta/d_w}$, $\alpha \in (0,d_w)$, $\beta \in (-\alpha, 0)$ or $\beta \in (0,d_w-\alpha)$, then Assumption \ref{ass:ass1} also holds.
\end{itemize}
}
\end{example}

An important consequence of the first part of assumption \eqref{eq:assum1} is that the process $X$ has symmetric and strictly positive transition densities given by
\begin{equation}\label{eq:subord}
p(t,x,y) = \int_0^{\infty} g(u,x,y) \eta_t({\rm d}u), \quad t>0, \quad x, y \in \cG.
\end{equation}
The second part of this condition guarantees that
\begin{equation}\label{eq:sup-of-g}
\sup_{x,y \in \cG} p(t,x,y) \leq c_3 c_0(t) < \infty,\end{equation}
that for each fixed $t >0$, $p(t,\cdot,\cdot)$ is a continuous function on $\cG \times \cG$, and for each fixed $x, y \in \cG$, $p(\cdot,x,y)$ is a continuous function on $(0,\infty)$.

By general theory of subordination (see, e.g., \cite[Chapter 12]{bib:SSV}) the process $X$ is a Feller process and, in consequence, a strong Markov process. It is also easy to check that by \eqref{eq:assum1} it has the strong Feller property.

Under the assumption \eqref{eq:log} we obtain the following regularity for suprema of the subordinate process.
It will be pivotal for our further investigations.

\begin{lemma} \label{lm:supr_sum}
Let the condition \eqref{eq:log} of Assumption \ref{ass:ass1} hold. Then for every $t>0$ and $a>1$ we have
\begin{equation}\label{eq:sup}
\sum_{M=1}^\infty \sup_{x\in \mathcal G} \pr_x[\sup_{s\leq t}d(X_s,x)>a^M]<\infty.
\end{equation}
\end{lemma}

\begin{proof}
Observe that for every $x \in \cG$, $r>0$ and $t>0$ we have, since $S_0(w)=0$ a.s.,
$$
\left\{w:\sup_{s\leq t}d(Z_{S_s(w)}(w),Z_{S_0(w)}(w))>r \right\} \subseteq \left\{w:\sup_{s\leq S_t(w)}d(Z_s(w),Z_0(w))>r \right\}.
$$
This and \cite[formula (3.21)]{bib:Bar} 
thus yield that for every $x \in \cG$
\begin{align*}
\pr_x[\sup_{s\leq t}d(X_s,x)>r] & \leq \pr_x[\sup_{s\leq S_t}d(Z_s,x)>r] \leq c_1 \int_0^{\infty} {\rm e}^{-c_2 \left(ru^{-1/d_w}\right)^{d_w/(d_w-1)}} \eta_t({\rm d}u) \\
& = c_1 \int_0^{r^{d_w/2}} + c_1 \int_{r^{d_w/2}}^{\infty}
\leq c_1 {\rm e}^{-c_2 r^{d_w/(2(d_w-1))}} + c_1 \eta_t(r^{d_w/2},\infty).
\end{align*}
By \eqref{eq:log}, the latter sum for $r=a^M$ is a term of a convergent series.
\end{proof}

For an open bounded set $D \subset \cG$ by $\tau_D:=\inf\left\{t \geq 0: X_t \notin D \right\}$ we denote the first exit time of the process $X$ from $D$. We will  need the following fact on the mean exit time for balls.

\begin{lemma} \label{lem:meanexittime} We have $\lim_{r \to 0^+} \sup_{x \in \cG} \ex_x \tau_{B(x,r)} = 0$.
\end{lemma}

\begin{proof}
As in \cite[Lemma 3.9(c)]{bib:Bar} or \cite[Prop. 2.14]{bib:BSS}
 we write, for any $t>0,$ $x\in\cG$
\begin{eqnarray*}
\mathbf P_x[{\tau_{B(x,r)}}>t]&\leq&\mathbf P_x[X_t\in B(x,r)]=\int_{B(x,r)} p(t,x,y)\,{\rm d}y\\
&\leq& m(B(x,r))\sup_{y\in\cG} p(t,x,y)\leq c_3 r^{d_f}c_0(t)=: a(r,t_0).
\end{eqnarray*}
From Markov property we get, for $k=1,2,...$
\[\mathbf P_x[\tau_{B(x,r)}>kt_0]\leq a(r,t_0)^k,\]
and, as long as $a(r,t_0)<1,$
\begin{eqnarray*}\mathbf E_x[\tau_{B(x,r)}]&=&\int_0^\infty \mathbf P_x[\tau_{B(x,r)}>t]{\rm d}t = \int_0^{t_0}\mathbf P_x[\tau_{B(x,r)}>t]{\rm d}t+t_0\sum_{k=1}^\infty \mathbf P_x[\tau_{B(x,r)}>kt_0]
\\&\leq& \frac{t_0}{1-a(r,t_0)}.
\end{eqnarray*}
Fix $\epsilon>0$ and $t_0={\epsilon}.$ For $r<r_0(\epsilon)= (2c_3c_0(\epsilon))^{-1/d_f})$ we get $\frac{1}{1-a(r,t_0)}<2,$ and thus for any $x\in\cG$ $\mathbf E_x[\tau_{B(x,r)}]\leq 2\epsilon,$ which completes the proof.
\end{proof}

By $\pr^t_{x,y}$ we denote the bridge measure with respect to process $X$ on $D([0,t],\cG)$, i.e. the measure concentrated on c\`adl\`ag paths of $X$ which start from $x \in \cG$ at time $0$ and end at $y \in \cG$ at time $t>0$. Subordinate process $X$ is a Feller process with sufficiently regular transition densities, and therefore in our case such a measure always exists (see \cite{bib:ChUB} and references therein). Formally, for every $x,y \in \cG$ and $t>0$ the measure $\pr^t_{x,y}$ is the conditional law of the process $(X_s)_{0 \leq s \leq t}$ given $X_t=y$ under $\mathbf P_x,$ such that for every $0 \leq s < t$ and $A \in \sigma(X_s: 0 \leq u \leq s)$ we have
\begin{align}\label{eq:bridge}
p(t,x,y) \pr^t_{x,y}(A) = \ex_x\left[\1_A p(t-s,X_s,y)\right].
\end{align}
{Moreover, an essential consequence of the Feller property of $X$ is that after performing the integration with respect to $m({\rm d} y)$, \eqref{eq:bridge} extends to $s=t$. Indeed, for every bounded Borel function $f$ and $A \in \sigma(X_s: 0 \leq u \leq t)$ we have
\begin{align}\label{eq:extbridge}
\ex_x \left[\1_A f(X_t)\right] = \int_{\cG} \ex^t_{x,y}\left[\1_A \right] f(y) p(t,x,y) m({\rm d} y), \quad x \in \cG.
\end{align}
For justification of \eqref{eq:extbridge} we refer to \cite[Section 3]{bib:ChUB}.}
Since $p(t,x,y)=p(t,y,x)$ for $x, y \in \cG$ and $t>0$, the processes $\left((X_s)_{0 \leq s \leq t}, \pr^t_{x,y}\right)$ and $\left((X_{t-s})_{0 \leq s \leq t}, \pr^t_{y,x}\right)$ are indentical in law. Below we will refer to this property as the symmetry of the brigde measure.

\subsubsection{The reflected process}
The processes we will be mostly working with are the reflected subordinate Brownian motions on $\mathcal G_M,$
defined by
\[X_t^M=\pi_M(X_t),\]
where $\pi_M$ is the projection $\mathcal G\to \mathcal G_M,$ described in Section \ref{sec:gasket}. Similar process based on the ordinary Brownian motion on $\mathcal G^*$
was studied in \cite{bib:KPP-PTRF}. Only $M=0$ was considered there, but for arbitrary $M$
the properties of the reflected Brownian motion are similar. Also, the construction of the reflected Brownian motion from the Brownian motion on $\mathcal G$ instead of $\mathcal G^*$ leads to the same process.  The process $Z_t^M=\pi_M(Z_t)$
has strictly positive and symmetric transition densities with respect to $m,$ which are given by the formula
\[g^M(t,x,y)=\left\{\begin{array}{ll}
\sum_{y'\in \pi_M^{-1}(y)} g(t,x,y'), & \mbox{when } x,y\in\mathcal G_M,\, y\notin \mathcal V_M\setminus\{0\},\\[2mm]
2\sum_{y'\in \pi_M^{-1}(y)} g(t,x,y'), & \mbox{when } x\in\mathcal G_M,\, y\in \mathcal V_M \backslash \left\{0\right\}.\end{array}
\right.\]
For each fixed $M \in \mathbb{Z}_+
$ the function $g^M(t,x,y)$ is jointly continuous in $(t,x,y)$ and symmetric in its space variables. This was proved in \cite[Lemma 4 and Lemma 7]{bib:KPP-PTRF} for $M=0$, but the same arguments directly extend to any $M \in \mathbb Z_+$. Similarly, we put
\[p^M(t,x,y)=\left\{\begin{array}{ll}
\sum_{y'\in \pi_M^{-1}(y)} p(t,x,y'), & \mbox{when } x,y\in\mathcal G_M,\, y\notin \mathcal V_M\setminus \{0\},\\[2mm]
2\sum_{y'\in \pi_M^{-1}(y)} p(t,x,y'), & \mbox{when } x\in\mathcal G_M,\, y\in \mathcal V_M \backslash \left\{0\right\}.\end{array}
\right.\]
It is an easy observation that the projection commutes with subordination, i.e. the formula
\begin{equation}\label{eq:proj-sub}
p^M(t,x,y)=\int_0^\infty g^M(u,x,y)\eta_t({\rm d}u),\;\;\;x,y\in \mathcal G_M, \;t>0,
\end{equation}
defines the transition densities of the symmetric Markov process $X^M$. To discuss further properties of the reflected process $X^M$ we need the following lemma.

\begin{lemma}
\label{lm:series} For each fixed $t>0$ one has:
\begin{itemize}
\item[(a)]    \begin{equation}\label{eq:gauss-pom} \sup_{x,y \in \mathcal G_M }  \sum_{\stackrel{y^{'} \in \pi^{-1}_M(y)}{y'\notin \mathcal G_{M+1}}} g(t,x, y') \leq  C  {2^{-Md_f}}\left(\frac{2^M}{t^{1/d_w}}\vee 1\right)^{d_f-\frac{d_w}{d_w-1}}{\rm e}^{-\tilde c_4 \left(\frac{2^M}{t^{1/d_w}}\vee 1\right)^{\frac{d_w}{d_w-1}}}, \quad M \in \Z,\end{equation}
with certain numerical constant $C>0$ and $\tilde c_4=\frac{c_4}{2^{d_w/(d_w-1)}}.$  In particular, there is a universal constant  $c_5>0$ such that
\begin{align}\label{eq:gmbound}
g^M(t,x,y) \leq c_5 {(t^{-d_s/2} \vee 2^{-Md_f})}, \quad x , y \in
\cG_M, \quad t>0, \quad M \in \Z.
\end{align}
\item[(b)] For
\begin{equation}\label{eq:summ}C(M,t) := \sup_{x,y \in \mathcal G_M }  \sum_{\stackrel{y^{'} \in \pi^{-1}_M(y)}{y'\notin \mathcal G_{M+1}}} p(t,x, y^{\prime}), \quad t>0, \quad M \in \Z, \end{equation}
one has
$$
\sum_{M=1}^{\infty} C(M,t) < \infty.
$$
In particular, $C(M,t) \to 0$ as $M \to \infty$.
\end{itemize}
\end{lemma}

\begin{proof}
Let $x, y \in \cG_M$. All the points $y'$ from the sums below are taken from the fiber of $y.$

\smallskip

(a)
One can write
\[\sum_{y'\notin\mathcal G_{M+1}} g(t,x,y')=\sum_{k\geq 1}\; \sum_{y'\in \mathcal G_{M+k+1}\setminus \mathcal G_{M+k}} p(t,x,y').\]
When $y'\in \mathcal G_{M+k+1}\setminus \mathcal G_{M+k},$ then $d(y',0)\geq 2^{M+k},$ and since $d(x,0)\leq 2^M,$ we have $d(y',x)\geq \frac{1}{2}2^{M+k}.$ The number of such points $y'$ is not bigger than the number of $M-$triangles in $ \mathcal G_{M+k+1}\setminus \mathcal G_{M+k},$ i.e. $2\cdot  {3^{k}}.$ Therefore from the subgaussian estimate on $g$ we get:
\[ \sum_{y'\notin\mathcal G_{M+1}} g(t,x,y') \leq \frac{2c_3}{ {2^{Md_f}}t^{d_s/2}}\sum_{k\geq 1} 3^{M+k}
{\rm e}^{-c_4\left(\frac{2^{M+k}}{2t^{1/d_w}}\right)^{\frac{d_w}{d_w-1}}}=(*).\]
Since for any $C,\gamma>0$ the function $x\mapsto {\rm e}^{-Cx^\gamma}$ is monotone decreasing, we can estimate the series above by an appropriate integral, getting that
\[(*) \leq \frac{\tilde c_3}{ {2^{Md_f}}t^{d_s/2}} \int_{3^M}^\infty {\rm e}^{-\tilde c_4\left(\frac{x^{1/d_f}}{t^{1/d_w}}\right)^{\frac{d_w}{d_w-1}}}\,{\rm d}x= \frac{\tilde c_3 d_f }{ {2^{Md_f}}}\int_{2^Mt^{-\frac{1}{d_w}}}^\infty y^{d_f-1} {\rm e}^{-\tilde c_4y^{\frac{d_w}{d_w-1}}}{\rm d} y,\]
where we have denoted $\tilde c_4=\frac{c_4}{2^{\frac{d_w}{d_w-1}}}.$
Using an elementary estimate
\begin{eqnarray*}\int_a^\infty y^\beta {\rm e}^{-\eta y^\gamma}{\rm d} y\leq C (a\vee 1)^{\beta-\gamma+1}
{\rm e}^{-\eta (a\vee 1)^\gamma}, \quad \eta,\beta,\gamma >0,
\end{eqnarray*}
we can write
\[\sum_{y'\notin \mathcal G_{M+1}} g(t,x,y') \leq C  {2^{-Md_f}} \left(\frac{2^M}{t^{1/d_w}}\vee 1\right)^{d_f-\frac{d_w}{d_w-1}}{\rm e}^{-\tilde c_4 \left(\frac{2^M}{t^{1/d_w}}\vee 1\right)^{\frac{d_w}{d_w-1}}},\]
with certain numerical constant $C>0,$ getting (\ref{eq:gauss-pom}).

To see \eqref{eq:gmbound}, first observe that, given $y\in\cG_M,$ there are at most 3 points
in $\pi^{-1}_M(y)\cap\cG_{M+1},$ so that, using \eqref{eq:gaussian} and \eqref{eq:gauss-pom} we get
\[\sup_{x,y\in\cG_M}g^M(t,x,y)\leq \frac{4c_3}{t^{d_s/2}}+  C  {2^{-Md_f}} \left(\frac{2^M}{t^{1/d_w}}\vee 1\right)^{d_f-\frac{d_w}{d_w-1}}{\rm e}^{-\tilde c_4 \left(\frac{2^M}{t^{1/d_w}}\vee 1\right)^{\frac{d_w}{d_w-1}}}.\]
When $t\geq 5^M,$ then this is bounded by  {$\frac{4c_3}{2^{Md_f}}+ \frac{C{\rm e}^{-\tilde c_4}}{2^{Md_f}}=:\frac{c}{2^{Md_f}}$}.
On the other hand, for $t\leq 5^M,$  denoting $a=\frac{2^M}{t^{1/d_w}}\geq 1$ we observe
\[a^{d_f-\frac{d_w}{d_w-1}}{\rm e}^{-\tilde c_t a^{\frac{d_w}{d_w-1}}}\leq a^{d_f}{\rm e}^{-\tilde c_4}= {\rm e}^{-\tilde c_4} \,\frac{ {2^{Md_f}}}{t^{d_f/d_w}}.\] Once we note that $d_f/d_w=d_s/2,$ this is exactly what is needed to get \eqref{eq:gmbound}.
\medskip

(b)
We now integrate the bound (\ref{eq:gauss-pom}) against the distribution of $S_t,$ getting
\begin{eqnarray}\label{eq:esti-pm}
\sum_{y'\notin \mathcal G_{M+1}} p(t,x,y') &\leq & C\int_0^{5^M}\left(\frac{2^M}{u^{1/d_w}}\right)^{d_f-\frac{d_w}{d_w-1}}{\rm e}^{-\tilde c_4 \left(\frac{2^M}{u^{1/d_w}}\right)^{\frac{d_w}{d_w-1}}}\eta_t({\rm d}u)
\\
&& + C\int_{5^M}^\infty {\rm e}^{-\tilde c_4}\eta_t({\rm d}u)\nonumber \\
&=:& A(M,t)+C \ {\eta_t(5^M,\infty)}=:C(M,t).\nonumber
\end{eqnarray}
In view of \eqref{eq:log} (see Remark \ref{rem:rem1} (1)), to get $\sum_M C(M,t)<\infty,$ it is enough to check that $\sum_MA(M,t)<\infty.$ To shorten the notation, write $\gamma=d_f-\frac{d_w}{d_w-1}.$
We have
\begin{eqnarray*}
\sum_{M=1}^\infty A_M(t) &\leq & \sum_{M=1}^\infty \int_0^{5^M} \left(\frac{2^M}{u^{1/d_w}}\right)^\gamma
{\rm e}^{-\tilde c_4 \left(\frac{2^M}{u^{1/d_w}}\right)^{\frac{d_w}{d_w-1}}}\eta_t({\rm d}u)\\
&=& \sum_{M=1}^\infty\sum_{n=1}^\infty \int_{5^{M-n-1}}^{5^{M-n}} \left(\frac{2^M}{u^{1/d_w}}\right)^\gamma
{\rm e}^{-\tilde c_4 \left(\frac{2^M}{u^{1/d_w}}\right)^{\frac{d_w}{d_w-1}}}\eta_t({\rm d}u)\\
&\leq & C \sum_M\sum_n 2^{n\gamma}{\rm e}^{-\tilde c_4 2^{\frac{nd_w}{d_w-1}}} {\eta_t(5^{M-n-1},5^{M-n}]}
\\
&=& \sum_n 2^{n\gamma}{\rm e}^{-\tilde c_4 2^{\frac{nd_w}{d_w-1}}}{\eta_t(5^{-n},\infty)} <\infty,
\end{eqnarray*}
which completes the proof.

\end{proof}

Thanks to the second part of the assumption \eqref{eq:assum1}, \eqref{eq:proj-sub} and \eqref{eq:gmbound}, the kernel $p^M(t,x,y)$ has the same continuity properties as $p(t,x,y)$. Also,
$$\sup_{x,y \in \cG} p^M(t,x,y) \leq c(M,t) < \infty.$$
Furthermore, it is easy to see that for each fixed $M \in \mathbb{Z}_+$ the process $X^M$ is Feller and strong Feller, and, therefore, again, we may and do consider the corresponding bridge process (see \cite{bib:ChUB}). Similarly as before, by $\pr^{M,t}_{x,y}$, $M \in \mathbb{Z}_+$, $t>0$, $x,y \in \cG_M$, we denote the bridge measure of the process $(X^M_s)_{0 \leq s \leq t}$ on $D([0,t], \cG_M)$, satisfying the usual bridge property: for every $0 \leq s < t$ and $A \in \sigma(X^M_s: 0 \leq u \leq s)$ we have
\begin{align}\label{eq:bridge-neum}
p^M(t,x,y) \pr^{M,t}_{x,y}(A) = \ex_x\left[\1_A p^M(t-s,X_s,y)\right].
\end{align}
{Similarly as for the process $X$ (see \eqref{eq:extbridge}), thanks to the Feller property of $X^M$, \eqref{eq:bridge-neum} extends to $s=t$ after performing the integration with respect to $m({\rm d} y)$. }Moreover, the bridge measures for the reflected and the ordinary subordinate process are related through the
following identity.

\begin{lemma}
\label{lm:rotation}
(a) For every $t>0$, $x,y \in \cG\setminus \mathcal V_M$, $M \in \Z$ and the set $A \in \cB(D[0,t],\cG_M)$ we have

\begin{equation}\label{eq:rot1}
p^M(t,\pi_M(x),\pi_M(y))\mathbf P^{M,t}_{\pi_M(x),\pi_M(y)}[A]=
\sum_{y'\in\pi_M^{-1}(\pi_M(y))} p(t,x,y')\mathbf
P^t_{x,y'}[\pi_M^{-1}(A)]\end{equation}

\medskip

(b) Consequently, for any $i=1,2,3$ and $x\in\cG_M,$
\begin{align*}
\sum_{x^{'} \in \pi^{-1}_{M}(x) } p(t,\pi^{-1}_{M,i}(x),x^{\prime})   \ex^t_{\pi^{-1}_{M,i}(x),x^{'}}\left[\pi^{-1}_M(A) \right]
= \sum_{x^{'} \in \pi^{-1}_{M}(x) } p(t,x,x^{\prime})  \ex^t_{x,x^{'}}\left[\pi^{-1}_M(A)\right].
\end{align*}
\end{lemma}

\begin{proof} (a) This is a consequence of
Theorem 3  in \cite{bib:KPP-PTRF}. For the Brownian density one has
\begin{equation}\label{eq:sgum}
\sum_{z'\in \pi_M^{-1}(z)}g(t,x,z')=
\sum_{z'\in \pi_M^{-1}(z)}g(t,y,z'),
\end{equation}
whenever $z\in\mathcal G_M$ and $x,y\in\mathcal G$ are such points that $\pi_M(x)=\pi_M(y)$
(the proof in \cite{bib:KPP-PTRF} is carried for $M=0$ only; the general case is similar).
Integrating this against $\eta_t$  we get identical property for $p(\cdot,\cdot,\cdot).$
Now, as it was done in \cite[Lemma 8]{bib:KPP-PTRF}, it is enough to check the relation (\ref{eq:rot1}) for cylindrical sets only. This is straightforward using property (\ref{eq:sgum}) for $p.$
To get (b) just observe that $\pi_M(\pi_{M,i}^{-1}(x))=x$ and then use (\ref{eq:rot1}).
\end{proof}

\noindent
In the sequel, we will  need the following trace type property.
\begin{lemma}\label{lm:ergodic}
For every $t>0$,
\begin{equation}\label{eq:sum}
\sum_{M=1}^\infty
\frac{1}{m(\mathcal G_M)} \int_{\mathcal G_M} \left|p(t,x, x) - p^M(t, x, x)\right|  m({\rm d}x) < \infty.
\end{equation}
In particular,
\begin{align}
\label{eq:onedim3}
\frac{1}{m(\mathcal G_M)} \int_{\mathcal G_M} \left|p(t,x, x) - p^M(t, x, x)\right|  m({\rm d}x) \to 0, \quad \text{as} \quad M \to \infty.
\end{align}
\end{lemma}

\begin{proof} We can write, with all the points $x'$ taken from the fiber of $x,$
\begin{eqnarray*}
\frac{1}{m(\mathcal G_M)}\int_{\mathcal G_M}|p^M(t,x,x)-p(t,x,x)|\,m({\rm d}x) &\leq&
\frac{2}{m(\mathcal G_M)}\int_{\mathcal G_M} \sup_{x'\in\mathcal G_{M+1}\setminus\mathcal G_M} p(t,x,x')\,m({\rm d}x)\\
&&+ \frac{1}{m(\mathcal G_M)}\int_{\mathcal G_M}\sum_{x'\notin \mathcal G_{M+1}}p(t,x,x')\,m({\rm d}x)\\
&=:& \mathcal A_M+\mathcal B_M.
\end{eqnarray*}
We have $\mathcal B_M\leq C(M,t)$ (recall $C(M,t)$ is defined in \eqref{eq:summ}), which is a term of a convergent series. Therefore it remains to prove that $\sum \mathcal A_M<\infty.$
We split the integral in $\mathcal A_M$ into two integrals: over $E_M^1=B(0, {2^M}\left(1-\frac{1}{M^2}\right))$ and over $E_M^2=\mathcal G_{M}\setminus E_M^1.$ Using (\ref{eq:kolnierzyk}), we can write
\[\frac{m(E_M^2)}{m(\mathcal G_M)}\leq \frac{c}{M^{2(d_f-1)}},\]  and since $p(t,x,x')\leq c_0(t),$  and $2(d_f-1)>1,$ the part of the expression corresponding to the integral over $E^2_M$ is a term
of a convergent series.

When $x\in E_1^M,$ then for all $x'\notin \mathcal G_M$ one has $d(x,x')> \frac{2^M}{M^2}.$ From the subgaussian estimates and the subordination formula one gets
\begin{eqnarray*}
p(t,x,x')& = & \int_0^\infty g(u,x,x')\eta_t({\rm d}u)\leq c_3\int_0^\infty\frac{1}{u^{d_s/2}}\,{\rm e}^{-c_4\left(\frac{2^M}{M^2\cdot u^{1/d_w}}\right)^{\frac{d_w}{d_w-1}}}\eta_t({\rm d}u)
\end{eqnarray*}
It follows
\[\frac{1}{m(\mathcal G_M)}\int_{E^1_M} \sup_{x'\in\mathcal G_{M+1}\setminus \mathcal G_M} p(t,x,x')\,m({\rm d}x) \leq c_3\int_0^\infty \frac{1}{u^{d_s/2}}\, {\rm e}^{-c_4\left(\frac{2^M}{M^2\cdot u^{1/d_w}}\right)^{\frac{d_w}{d_w-1}}}\eta_t({\rm d}u).
\]
Summing these integrals up, we obtain that
\[\sum_{M\geq 1} \mathcal A_M \leq \sum_{M\geq 1}\frac{c}{M^2}+ c_3\int_0^\infty \frac{1}{u^{d_s/2}} \sum_{M\geq 0} {\rm e}^{-c_4\left(\frac{2^{M+1}}{(M+1)^2u^{1/d_w}}\right)^{\frac{d_w}{d_w-1}}}\eta_t({\rm d}u).
\]
We now take care of the sum under the integral sign. Again, we compare it with appropriate integrals. Observe that for $M\geq 3$
\begin{eqnarray*}
\int_{2^M/M^2}^{2^{(M+1)}/(M+1)^2} \frac{1}{x}\,{\rm e}^{-c_4\left(\frac{x}{u^{1/d_w}}\right)^{\frac{d_w}{d_w-1}}}{\rm d}x& \geq &
\left(\frac{2^{M+1}}{(M+1)^2}-\frac{2^M}{M^2}\right)\frac{(M+1)^2}{2^{M+1}} \,
{\rm e}^{-c_4\left(\frac{2^{M+1}}{(M+1)^2\cdot u^{1/d_w}}\right)^{\frac{d_w}{d_w-1}}}\\
&\geq & c  \,
{\rm e}^{- c_4\left(\frac{2^{M+1}}{(M+1)^2\cdot u^{1/d_w}}\right)^{\frac{d_w}{d_w-1}}}.
\end{eqnarray*}
It follows that
\begin{eqnarray*}\sum_{M\geq 0}  \,
{\rm e}^{-c_4\left(\frac{2^{M+1}}{(M+1)^2\cdot u^{1/d_w}}\right)^{\frac{d_w}{d_w-1}}}&\leq& 3+c \int_{8/9}^\infty {\rm e}^{-c_4 \left(\frac{x}{u^{1/d_w}}\right)^{\frac{d_w}{d_w-1}}}\frac{{\rm d}x}{x},
\end{eqnarray*}
and, consequently,
\[\sum_{M\geq 1}\mathcal A_M\leq c\left(1+ \int_0^\infty \frac{\eta_t({\rm d}u)}{u^{d_s/2}} + \int_0^\infty\left[ \int_{8/9}^\infty {\rm e}^{-c_4 \left(\frac{x}{u^{1/d_w}}\right)^{\frac{d_w}{d_w-1}}}\frac{{\rm d}x}{x}\right]\frac{\eta_t({\rm d}u)}{u^{d_s/2}}\right).
\]
According to Assumption \ref{ass:ass1} (see \eqref{eq:assum1}), it suffices to show that the last double integral is convergent.
By Fubini, we see that it is equal to

\begin{align*}
\int_{8/9}^\infty \left[\left(\int_0^x + \int_x^{\infty} \right)  {\rm e}^{-c_4 \left(\frac{x}{u^{1/d_w}}\right)^{\frac{d_w}{d_w-1}}}\frac{\eta_t({\rm d}u)}{u^{d_s/2}}\right]\frac{{\rm d}x}{x}
\leq \int_{8/9}^\infty {\rm e}^{-c_4 x} \frac{{\rm d}x}{x} \ \cdot \ \int_0^{\infty} \frac{\eta_t({\rm d}u)}{u^{d_s/2}} + c \int_{8/9}^\infty \eta_t(x,\infty) \frac{{\rm d}x}{x}.
\end{align*}
Again, by Assumption \ref{ass:ass1}, the last two integrals above are convergent. We are done.

\end{proof}

\subsection{Processes perturbed by Schr\"odinger potentials}

\subsubsection{Nonrandom Feynman-Kac semigroups}

We say that a Borel function $V$ is in Kato class $\cK^X$ related to the process $X$ if
\begin{align}
\label{eq:Kato}
\lim_{t \searrow 0} \sup_{x \in \cG} \int_0^t \ex_x |V(X_s)|\,{\rm d}s = 0.
\end{align}
Also, $V \in \cK^X_{\loc}$ (local Kato class), when $\1_B V \in \cK^X$ for every ball $B \subset \cG$. {Obviously, $L^{\infty}_{\loc}(\cG) \subset \cK^X_{\loc}$. Furthermore, it is a general fact that $\cK_{\loc}^X \subset L^1_{\loc}(\cG,m)$. It is also useful to note that under the condition $\int_0^1 c_0(t) {\rm d}t < \infty$ (recall that the constant $c_0(t)$ was defined in $\eqref{eq:assum1}$), in fact one has $\cK_{\loc}^X = L^1_{\loc}(\cG,m)$. Indeed, for $V \in L^1_{\loc}(\cG,m)$ and an arbitrary bounded Borel set $B \subset \cG$ we get, using the subordination formula (\ref{eq:subord}), estimate (\ref{eq:gaussian}), and the assumption on $c_0$:
$$
\sup_{x \in \cG} \int_0^t \ex_x |V(X_s)\1_B(X_s)|{\rm d}s \leq c \int_0^t \int_0^{\infty} u^{-d_s/2} \eta_s({\rm d}u) {\rm d}s \cdot \int_B |V(y)| m({\rm d}y) \leq c_1 \int_0^t c_0(s) {\rm d}s\to 0
$$
as $t \to 0^+$. Hence $V \in  \cK^X_{\loc}$. }

Under our conditions on the process $X$, formula \eqref{eq:Kato} can be rewritten as (see \cite[Theorem 1]{bib:Z})
\begin{align}
\label{eq:Kato1}
\lim_{\varepsilon \searrow 0} \sup_{x \in \cG} \ex_x \int_0^{\tau_{B(x,\varepsilon)}} |V(X_s)|{\rm d}s = 0.
\end{align}
The condition \eqref{eq:Kato1} is always very useful. For instance, when $V \in \cK^X_{\loc}$ then by using \eqref{eq:Kato1} one can easily show that for every $M \in \Z$ we also have $V_M \in{\cK^X}$, where $V_M$ is the usual periodization of $V$, i.e. $V_M(x):=V(\pi_M(x))$, $x \in \cG$.

Under the condition $V \in \cK^X_{\loc}$, we may define the Feynman-Kac semigroups related to the killed and the reflected process in $\cG_M$, $M \in \Z$. Let
\begin{align}
\label{def:sem-dir}
T_t^{D,M} f(x) = \ex_x \left[{\rm e}^{-\int_0^t V(X_s) {\rm d}s} f(X_t); t<\tau_{\cG_M}\right], \quad f \in L^2(\cG_M,m), \quad M \in \Z, \quad t>0,
\end{align}
and
\begin{align}
\label{def:sem-neu}
T_t^{N,M} f(x) = \ex^M_x \left[{\rm e}^{-\int_0^t V(X^M_s) {\rm d}s} f(X^M_t)\right], \quad f \in L^2(\cG_M,m), \quad M \in \Z, \quad t>0.
\end{align}
It is not very difficult to check that both semigroups $(T_t^{D,M})$ and $(T_t^{N,M})$ are ultracontractive. Since for every $t>0$, both operators $T_t^{D,M}$ and $T_t^{N,M}$ are also symmetric and bounded on $L^2(\cG_M)$, they admit measurable, symmetric and bounded kernels (see \cite[Theorem A.1.1 and Corollary A.1.2]{bib:S}), i.e.
\begin{align}
\label{def:sem-dir-kernel}
T_t^{D,M} f(x) = \int_{\cG_M} u^M_D(t,x,y) f(y) m({\rm d}y), \quad f \in L^2(\cG_M,m), \quad M \in \Z, \quad t>0,
\end{align}
\begin{align}
\label{def:sem-neu-kernel}
T_t^{N,M} f(x) = \int_{\cG_M} u^M_N(t,x,y) f(y) m({\rm d}y), \quad f \in L^2(\cG_M,m), \quad M \in \Z, \quad t>0.
\end{align}
{For verification of all basic properties of Feynman-Kac semigroups for Markov processes, including those listed above, we refer the reader to \cite[Sections 3.2 and 3.3]{bib:CZ}}.

{By \eqref{eq:extbridge} and its counterpart for the process $X^M$}, one obtains the following very useful bridge representations for the kernels. We have
\begin{align}
\label{def:sem-dir-kernel-bridge}
u^M_D(t,x,y) = p(t,x,y) \ \ex^t_{x,y}\left[{\rm e}^{-\int_0^t V(X_s) {\rm d}s} ; t<\tau_{\cG_M}\right],  \quad M \in \Z, \quad x,y \in \cG_M, \quad t>0,
\end{align}
\begin{align}
\label{def:sem-neu-kernel-bridge}
u^M_N(t,x,y) = p^M(t,x,y) \ \ex^{M,t}_{x,y} \left[{\rm e}^{-\int_0^t V(X^M_s) {\rm d}s}\right],  \quad M \in \Z, \quad x,y \in \cG_M, \quad t>0.
\end{align}

To shorten the notation, we will write $e_V(t):={\rm e}^{-\int_0^t V(\cdot){\rm d}s}$, $t>0$, for the Feynman-Kac functionals for both processes $X$ and $X^M$.

{Generators of semigroups $(T_t^{D,M})$ and $(T^{N,M}_t)$ will be denoted by $A^{M,D}$ and $A^{M,N}$, respectively. By analogy to the classical case, the operators $-A^{M,D}$ and $-A^{M,N}$ will be called the \emph{generalized Schr\"odinger operators} corresponding to the generator of the process $X$. The first one,  $-A^{M,D}$, is related to the killed process and therefore, in fact, it is a Schr\"odinger operator based on the generator of the process $X$ with Dirichlet (outer) conditions. Similarly, $-A^{M,N}$ may be seen as the Schr\"odinger operator based on the `Neumann' generator of this process. Indeed, the process $X^M$  appears via the subordination of the reflected Brownian motion (recall \eqref{eq:proj-sub}) and it can be regarded as a jump counterpart of the process `reflected on exiting the set $\cG_M$'. Let us emphasize, however, that when the process has discontinuous trajectories, then there is no
canonical definition of the `reflected process'  and there are several possible ways of constructing it. }

{As we pointed out above, for every $t>0$ and $M\in\mathbb Z_+$, both kernels $u^M_D(t,x,y)$ and $u^M_N(t,x,y)$ are bounded. Since also $m(\cG_M)<\infty$ for all $M \in \mathbb Z_+$, all operators $T_t^{D,M}$ and $T^{N,M}_t$ are of Hilbert-Schmidt type. Therefore, the spectra of the related Schr\"odinger operators, $-A^{M,D}$ and $-A^{M,N}$, consist only of eigenvalues of finite multiplicity having no accumulation points, and they can be ordered as $\lambda_1^{D,M}\leq \lambda_2^{D,M}\leq ...$
and $\lambda_1^{N,M}\leq \lambda_2^{N,M}\leq ...$. It is known that the corresponding eigenfunctions $(\phi^{D,M}_n)_{n\geq 1}$ (resp. $(\phi^{N,M}_n)_{n\geq 1}$) form a complete orthonormal system in $L^2(\cG,m)$. }

\medskip

\subsubsection{Random potentials.}
{In the sequel, we will consider a more general case, when the potential $V$ is not a deterministic function.
Let $(\Omega,\mathcal M, \qpr)$ be a probability space and $V(x,\omega)$ -- a real-valued function
on $\cG \times \Omega$ such that $V(x,\cdot)$ is measurable for each fixed $x \in \cG$ and
$V(\cdot,\omega) \in \cK^X_{\loc}$ for $\qpr$-almost all $\omega \in \Omega$.
Such a function $V$ is called a \emph{random potential} or a \emph{random field}.}

{For a random potential $V(\cdot, \omega)$ which $\mathbb Q$-a.s. belongs to the local Kato class $\mathcal K^X_{\rm loc}$, we consider
 random Feynman-Kac semigroups $(T^{D,M,\omega}_t)_{t\geq 0}$ and
$(T^{N,M,\omega}_t)_{t\geq 0},$  given by
\eqref{def:sem-dir}-\eqref{def:sem-neu}. Both semigroups consist of Hilbert-Schmidt operators, the totality of eigenvalues of the corresponding generalized random Schr\"odinger operators $-A^{D,M,\omega}$ and $-A^{N,M,\omega}$ can be ordered as $\lambda_1^{D,M}(\omega)\leq \lambda_2^{D,M}(\omega)\leq ...$ and $\lambda_1^{N,M}(\omega)\leq \lambda_2^{N,M}(\omega)\leq ...$, respectively. }

{The basic objects we consider are} the random empirical measures on $\mathbb R_+$ based on these spectra, normalized by the volume of $\cG_M$:
\begin{equation}\label{eq:el-d}
l_M^D(\omega):= \frac{1}{m(\cG_M)} \sum_{n=1}^{\infty}
\delta_{\lambda_n^{D,M}(\omega)}
\end{equation}
and
\begin{equation}\label{eq:el-n}
l_M^N(\omega):= \frac{1}{m(\cG_M)} \sum_{n=1}^{\infty}
\delta_{\lambda_n^{N,M}(\omega)}.
\end{equation}
In this paper we are interested in the convergence of these measures as $M\to\infty.$

Our main results are obtained for the restricted class of Poissonian potentials which are defined below. Let
\begin{align}
\label{eq:poiss0} V(x,\omega):= \int_{\cG} W(x,y)
\mu^{\omega}({\rm d}y),
\end{align}
where $\mu^\omega$ is the random counting measure corresponding to
the Poisson point process on $\cG$,  with intensity $\nu {\rm d}m$,  $\nu>0,$ defined on a
probability space $(\Omega,\mathcal M, \mathbb{Q})$),  and $W:\cG \times \cG \to \R$ is a
measurable, nonnegative profile function.  Throughout the paper we assume that the Poisson process and the Markov process $X$ are independent.

Now we list and discuss regularity {assumptions} concerning the profile function $W$.
\bigskip

\begin{itemize}
\item[\textbf{(W1)}]
$W\geq 0$, $W(\cdot,y) \in \cK_{\loc}^X$ for every $y \in \cG$ and there exists a function $h \in L^1(\cG,m)$ such that $W(x,y) \leq h(y)$, whenever $d(y,0)\geq 2d(x,0)$.
\end{itemize}

\bigskip

\noindent {Our final Theorems  \ref{th:meanlimit} and \ref{th:main-fractal}, addressing the problem of convergence of measures $l_M^D(\omega)$, $l^{N}_M(\omega)$ for Poissonian potentials of the form \eqref{eq:poiss0}}, require additional assumptions

\bigskip

\begin{itemize}
\item[\textbf{(W2)}] $\sum_{M=1}^{\infty} \sup_{x \in \cG} \int_{B(x,2^{M/4})^c} W(x,y) dm(y) < \infty$
\end{itemize}

\bigskip

and

\bigskip

\begin{itemize}
\item[\textbf{(W3)}] there is $M_0 \in \Z$ such that
\begin{equation}\label{eq:wu}
\sum_{y{'} \in \pi_M^{-1}(\pi_M(y))} W(\pi_M(x),y{'}) \leq \sum_{y{'}
\in \pi_M^{-1}(\pi_M(y))} W(\pi_{M+1}(x),y{'}), \quad x, y \in \cG,
\end{equation}
for every $M \in \Z$, $M \geq M_0$.
\end{itemize}

\bigskip

The following proposition asserts that under the condition {\bf (W1)} the function $V(\cdot,\omega)$ given by \eqref{eq:poiss0} is a well defined, locally Kato-class potential for almost all $\omega \in \Omega$.

\begin{proposition} \label{prop:Wkato} Let $V$ be a Poissonian potential whose profile $W$ satisfies {\bf (W1)}. Then $0 \leq V(\cdot,\omega) \in \cK_{\loc}^X$ for $\qpr$-almost all $\omega \in \Omega$.
\end{proposition}
\begin{proof}
First observe that thanks to \eqref{eq:Kato1}, it is enough to show that for every $R>0$ there is a measurable $\Omega_R \subset \Omega$ with $\qpr(\Omega_R)=1$ such that
\begin{align}\label{eq:toshow}
\lim_{\varepsilon \searrow 0} \sup_{x \in \cG} \ex_x \int_0^{\tau_{B(x,\varepsilon)}} |V_R(X_s,\omega)|{\rm d}s = 0,
\end{align}
for every $\omega \in \Omega_R$, where $V_R:=V \1_{B(0,R)}$. Since
$$
\mathbb{E_Q} \int_{\cG} h(y) \mu^{\omega}({\rm d}y) = \int_{\cG} h(y) m({\rm d}y) < \infty,
$$
there exists a measurable set $\Omega_1 \subset \Omega$ with $\qpr(\Omega_1)=1$ such that $h \in L^1(\cG,\mu_{\omega})$ for every $\omega \in \Omega_1$. Also, let
$$
\Omega_2 : = \left\{\omega \in \Omega: \text{finitely many Poisson points fell onto} \ B(0,2R) \right\}.
$$
By the definition of the cloud, $\qpr({\Omega_2})=1$, so $\Omega_R = \Omega_1 \cap \Omega_2$ is of full measure. We will prove that for every $\omega \in \Omega_R$ the condition \eqref{eq:toshow} holds.

For given $\omega \in \Omega_R$, denote by $\{y_i(\omega)\}_i$  the
realization of the cloud. Then for every $x \in \cG$ we have
$$
|V_R(x,\omega)| \leq \1_{B(0,R)}(x) \left(\sum_{y_i(\omega) \in B(0,2R)}  W(x,y_i(\omega)) + \int_{B(0,2R)^c} h(y) \mu^{\omega}({\rm d}y)\right),
$$
and, in consequence,
\begin{align*}
\ex_x \int_0^{\tau_{B(x,\varepsilon)}} |V_R(X_s,\omega)|{\rm d}s & \leq \sum_{y_i(\omega) \in B(0,2R)} \ex_x \int_0^{\tau_{B(x,\varepsilon)}}  \left[W(X_s,y_i(\omega))\1_{B(0,R+1)}(X_s)\right]{\rm d}s \\ & \ \ \ \ \ \ \ \ \ \ \ \ \ \ \ \ + \ \ex_x \tau_{B(x,\varepsilon)} \cdot \int_{B(0,2R)^c} h(y) \mu^{\omega}({\rm d}y),
\end{align*}
where the  sums above are taken over all Poisson points that fell onto $B(0,2R)$ (there is a finite number of them).  Since  $W(\cdot,y_i(\omega)) \in \cK_{\loc}^X$ for all $y_i(\omega)$, and $\int_{B(0,2R)^c} h(y) \mu^{\omega}({\rm d}y) < \infty$ for each $\omega \in \Omega_R$, we obtain by \eqref{eq:Kato1} (applied to $W$) and Lemma \ref{lem:meanexittime} that
\begin{align*}
\lim_{\varepsilon \searrow 0} \sup_{x \in \cG} & \,\ex_x \int_0^{\tau_{B(x,\varepsilon)}} |V_R(X_s,\omega)|{\rm d}s \\ &  \leq \sum_{y_i(\omega) \in B(0,2R)} \lim_{\varepsilon \searrow 0} \sup_{x \in \cG} \ex_x \int_0^{\tau_{B(x,\varepsilon)}} \left[ W(X_s,y_i(\omega))\1_{B(0,R+1)}(X_s)\right]{\rm d}s \\ & \ \ \ \ \ \ \ \ \ \ \ \ \ \ \ + \ \lim_{\varepsilon \searrow 0} \sup_{x \in \cG} \ex_x \tau_{B(x,\varepsilon)} \cdot \int_{B(0,2R)^c} h(y) \mu^{\omega}({\rm d}y) \ = \ 0.
\end{align*}
{Hence the proof of the proposition is complete.}
\end{proof}
\noindent
 {Note that the condition $\int_{\cG}W(\cdot,y) m({\rm d}y) \in L^1_{\loc}(\cG,m)$ immediately implies that $V(\cdot,\omega) \in L^1_{\loc}(\cG,m)$, $\qpr$-almost surely. By using this implication, one can give another sufficient condition for the property $V(\cdot,\omega) \in \cK_{\loc}^X$, $\qpr$-almost surely. Indeed, as we pointed out in the previous subsection, if $\int_0^1 c_0(t) dt < \infty$ holds, then $L^1_{\loc}(\cG,m) = \cK^X_{\loc}$. Therefore, under the assumption $\int_0^1 c_0(t) dt < \infty$, the condition
$$\int_{\cG}W(\cdot,y) m({\rm d}y) \in L^1_{\loc}(\cG,m) \quad \text{is  sufficient for} \quad V(\cdot,\omega) \in \cK_{\loc}^X, \ \qpr \text{-a.s.} .$$
}

\smallskip

The condition {\bf (W3)} involves the specific geometry of the gasket and is essentially different from remaining conditions {\bf (W1)} -- {\bf (W2)} which are analytic. As we will see, it will be decisive for our convergence problem. Examples of profile functions $W$ satisfying all our assumptions {\bf (W1)} -- {\bf (W3)} will be discussed in Section \ref{sec:ex}.

We finish this section by giving the following exponential formula. For every measurable and nonnegative function $f$ on $\cG$ it holds that
\begin{align}
\label{eq:underQ}
\ex_{\qpr} \left[{\rm e}^{-\int_{\cG} f(y) \mu^{\omega}({\rm d}y)}\right] = {\rm e}^{-\nu \int_{\cG}\left(1-{\rm e}^{-f(y)} \right) m({\rm d}y)}.
\end{align}
This formula will be a very important tool below (for its Euclidean counterpart we refer the reader to \cite[p. 433]{bib:Okura81}).
In particular, it yields the
representation for the averaged Feyman-Kac functional for the Poissonian potential $V$ with nonnegative profile $W$. Indeed, by taking $f(y)= \int_0^t W(X_s,y) ds$, for every $x \in \R^d$ and $\pr_x$-almost all $w$, we have
$$
\ex_{\qpr} \left[{\rm e}^{-\int_0^t V(X_s(w), \omega) {\rm d}s }\right] = {\rm e}^{-\nu \int_{\cG}\left(1-{\rm e}^{-\int_0^t W(X_s(w),y) {\rm d}s}\right) m({\rm d}y)}.
$$
 {For the reader's convenience, we give a short justification of \eqref{eq:underQ}. Suppose first that $f(x)=c \1_{A}(x)$ for some Borel set $A \subset \cG$ and $c>0$. Recall that $\mu^{\omega}(A)$ is a Poisson random variable with parameter $\nu m(A)$ on a probability space $(\Omega, \mathcal M, \qpr)$. With this we have
$$
\ex_{\qpr} \left[{\rm e}^{-\int_{\cG} f(y) \mu^{\omega}({\rm d}y)}\right] = \ex_{\qpr} \left[{\rm e}^{-c\mu^{\omega}(A)}\right] = {\rm e}^{-\nu m(A)\left(1-{\rm e}^{-c} \right)} = {\rm e}^{-\nu \int_{\cG}\left(1-{\rm e}^{-f(y)} \right) m({\rm d}y)}.
$$
Since for a family of pairwise disjoint Borel sets $A_1, ..., A_n \subset \cG$ the random variables $\mu^{\omega}(A_1)$, ..., $\mu^{\omega}(A_n)$ are independent, the above formula directly extends to  simple functions and then, by a standard approximation argument, it also holds for nonnegative measurable functions.
}

\section{Convergence}\label{sec:conv}

Our goal is to establish that the random measures $l_M^D(\omega)$
and $l^{N}_M(\omega),$ defined by (\ref{eq:el-d}), (\ref{eq:el-n}),  vaguely converge to a common limit $l,$
 which is a nonrandom measure on $\mathbb R_+.$  This measure is called the
 {\em integrated density of states}.

 We shall consider the Laplace transforms of $l_M^D(\omega),$ $l_M^N(\omega):$

\begin{eqnarray*}
L_M^D(t,\omega)&:=& \int_0^{\infty} {\rm e}^{-\lambda t} {\rm d}l_M^D(\omega)(t) =
\frac{1}{m(\cG_M)} \sum_{n=1}^{\infty} {\rm e}^{-\lambda_n^{D,M}(\omega) t} =
\frac{1}{m(\cG_M)} \tr T_t^{D,M,\omega}
\\
&=& \frac{1}{m(\cG_M)} \int_{\cG_M} p(t,x,x) \ex^t_{x,x}\left[e_{V}(t); t < \tau_{\cG_M} \right] m({\rm d} x)
\end{eqnarray*}
and
\begin{eqnarray*}
L_M^N(t,\omega)&:=& \int_0^{\infty} {\rm e}^{-\lambda t} {\rm d}l_M^N(\omega)(t) =
\frac{1}{m(\cG_M)} \sum_{n=1}^{\infty} {\rm e}^{-\lambda_n^{N,M}(\omega) t} =
\frac{1}{m(\cG_M)} \tr T_t^{N,M,\omega}\\
&=& \frac{1}{m(\cG_M)} \int_{\cG_M}
p^M(t,x,x) \ex^{M,t}_{x,x}\left[e_{V}(t)\right] m({\rm d}x).
\end{eqnarray*}

   First we show that the expectations
$\mathbb{E}_\qpr L_M^D(t,\omega)$ and $\mathbb
E_{\qpr}L_M^N(t,\omega)$ converge for every $t$ to a common limit $L(t).$

\subsection{Common limit of points of $L^N_M (t,\omega)$ and $L^D_M(t,\omega)$}
Our first result asserts that when the nonnegative (general, not necessarily Poissonian) random field $V$ belongs
to $\cK^X_{\loc}$, then $L^N_M (t,\omega)$ and $L^D_M(t,\omega)$ share the limit points in $L^2(\Omega,\qpr).$ In fact, we show more than that.

\begin{proposition}
\label{prop:dir_neum}
Let $0 \leq V(\cdot,\omega) \in \cK^X_{\loc}$ for $\qpr$-almost all $\omega$. Then for every $t>0$, $L^N_M (t,\omega)$
and $L^D_M(t,\omega)$ 
satisfy
$$
\sum_{M=1}^\infty
\mathbb{E}_{\qpr} \left(L^N_M (t,\omega) - L^D_M(t,\omega)\right)^2 <\infty.
$$
In particular,
\[\lim_{M\to\infty}\mathbb{E}_{\qpr}\left|L^N_M (t,\omega) - L^D_M(t,\omega)\right|^2=0.\]
\end{proposition}

\begin{proof}
First note that since $e_V(t) = e_{V_M}(t)$ on $\left\{t<\tau_{\cG_M}\right\}$ for $M \in \mathbb{Z}_+$, we have
\begin{equation}
\begin{split}
\label{eq:dir}
L^D_M (t,\omega) &= \frac{1}{m(\cG_M)} \int_{\cG_M} p(t,x,x) \ex^t_{x,x}\left[e_{V_M}(t); t < \tau_{\cG_M} \right] m({\rm d}x) \\
&= \frac{1}{m(\cG_M)} \int_{\cG_M} p(t,x,x) \ex^t_{x,x}\left[e_{V_M}(t)\right] m({\rm d}x)  \\ & \ \ \ \ \ \ -  \frac{1}{m(\cG_M)} \int_{\cG_M} p(t,x,x) \ex^t_{x,x}\left[e_{V_M}(t); t \geq \tau_{\cG_M} \right] m({\rm d}x).
\end{split}
\end{equation}
Simultaneously, by Lemma \ref{lm:rotation} (see \eqref{eq:rot1}), we  get
\begin{equation}
\begin{split}
\label{eq:neum} L^N_M (t,\omega) &= \frac{1}{m(\cG_M)} \int_{\cG_M}
p^M(t,x,x) \ex^{M,t}_{x,x}\left[e_{V}(t)\right] m({\rm d}x) \\ &=
\frac{1}{m(\cG_M)} \int_{\cG_M} p(t,x,x)
\ex^t_{x,x}\left[e_{V_M}(t)\right] m({\rm d}x) \\ & \ \ \ \ \ \ +
\frac{1}{m(\cG_M)} \int_{\cG_M} \sum_{x \neq x^{'} \in \pi^{-1}_M(x) }
p(t,x,x^{\prime}) \ex^t_{x,x^{'}}\left[e_{V_M}(t)\right] m({\rm d}x).
\end{split}
\end{equation}
We see that by \eqref{eq:dir}, \eqref{eq:neum} and the fact that $V \geq 0$, we get
$$
0 \leq L^N_M (t,\omega) - L^D_M (t,\omega) \leq  R_{1,M}(t) +  R_{2,M}(t),
$$
with
\begin{eqnarray}\label{eq:r-jeden}
R_{1,M}(t) &=& \frac{1}{m(\cG_M)} \int_{\cG_M} p(t,x,x) \pr^t_{x,x}\left[t \geq \tau_{\cG_M} \right] m({\rm d}x),
\\
\label{eq:r-dwa}
R_{2,M}(t) &= &\frac{1}{m(\cG_M)} \int_{\cG_M} \sum_{x \neq x^{'} \in \pi^{-1}_M(x) } p(t,x,x^{\prime}) m({\rm d}x).
\end{eqnarray}
(Note that these bounds do not depend on $\omega$). We can write, using (\ref{eq:sup-of-g})
and (\ref{eq:esti-pm}):
 \[(R_{1,M}(t) +  R_{2,M}(t))^2 \leq 2(R_{1,M}^2(t) +  R_{2,M}^2(t))\leq
 2c_3c_0(t) R_{1,M}(t)+ 2(c_3c_0(t))+C(M,t))R_{2,M}(t) \] for every $t>0$ and $M \in \mathbb{Z}_+.$ Since $C(M,t)$ remains bounded for $M\in \mathbb{Z}_+$ (in fact, it goes to zero as $M \to \infty$),  it is enough to check that for every fixed $t>0$ both members $R_{1,M}(t)$ and $R_{2,M}(t)$
 are terms of convergent series.

For the process starting from $x \in B(0,2^M-2^{M/2})$ one has
$\left\{t \geq \tau_{\cG_M}\right\} \subset \left\{\sup_{0< s \leq
t} d(x,X_s)
> 2^{M/2} \right\}.$ Moreover,
\begin{equation}\label{eq:zn}\left\{\sup_{0< s \leq
t} d(X_0,X_s)
> 2^{M/2} \right\}\subset \left\{\sup_{0< s \leq
t/2} d(X_0,X_s)
> 2^{M/2} \right\}\cup \left\{\sup_{t/2< s \leq
t} d(X_0,X_s)
> 2^{M/2} \right\},\end{equation}
and from Lemma \ref{lem:kolnierzyk} we have $m(\cG_M\setminus B(0,2^M-2^{M/2}))\leq c 2^{\frac{M}{2}(d_f+1)}$.
Using these facts, the bridge symmetry, and \eqref{eq:bridge}, we get
\begin{eqnarray*}
\int_{\cG_M} p(t,x,x)\pr^t_{x,x}\left[t \geq \tau_{\cG_M} \right]
 m({\rm d}x)&\leq& c(t) m(\cG_M\setminus B(0,2^M-2^{M/2}))\\
 && \ \ \ \ \ +\int_{B(0,2^M-2^{M/2})}p(t,x,x)\pr^t_{x,x}\left[t \geq \tau_{\cG_M} \right]
 m({\rm d}x)\\ &\leq & c(t) {2^{\frac{M}{2}\, (d_f+1)}}
 +2m(\cG_M) \sup_{x \in \cG}  \pr_x\left[\sup_{0< s \leq t/2} d(x,X_s) > 2^{M/2}\right] .
 \end{eqnarray*}
 Therefore
\begin{align*}
R_{1,M}(t) & \leq  c(t) 2^{{-\frac{M}{2}\,(d_f-1)}} + 2 \
\sup_{x \in \cG}  \pr_x\left[\sup_{0< s \leq t/2} d(x,X_s) >
2^{M/2}\right],
\end{align*}
and, by Lemma \ref{lm:supr_sum}, $R_{1,M}(t)$ is a term of a convergent series.

To estimate $R_{2,M}(t)$ it is enough to observe that
\begin{align*}
R_{2,M}(t) =\frac{1}{m(\cG_M)}\int_{\cG_M} (p^M(t,x,x)-p(t,x,x))m({\rm d}x).
\end{align*}
By Lemma \ref{lm:ergodic} (see \eqref{eq:sum}), this also is the term of a convergent series. This completes the proof.
\end{proof}

For more clarity we decided to prove the above proposition for $V \geq 0$ only. However, our argument can be directly modified to give  the same result in a much more general case.
The following remark asserts that Proposition \ref{prop:dir_neum} above remains true for a large class of signed random fields. Recall that $V_M(x)=V(\pi_M(x))$, $M \in \Z$, $x \in \cG$.

\begin{remark}
Assume that a random field $V=V^{+} - V^{-}$ (where $V^{+}$ and $V^{-}$ are, respectively, positive and negative parts of $V$) is such that $V^{+} \in \cK^X_{\loc}$, $V^{-} \in \cK^X$, for $\qpr$-almost all $\omega \in \Omega$, and for every $t>0$,
\begin{align}
\label{eq:M-condition}
\sup_{M \in \Z} \sup_{x \in \cG} \mathbb{E_Q} \times \ex_x \left[\exp\left(4 \int_0^t V^{-}_M(X_s,\omega) {\rm d}s\right) \right] < \infty.
\end{align}
Then the assertion of Proposition \ref{prop:dir_neum} also holds. Clearly, the condition \eqref{eq:M-condition} is satisfied if, e.g., $V^{-}$ is bounded above by the same finite constant for $\qpr$-almost all $\omega \in \Omega$.
\end{remark}

\subsection{Convergence of expectations of $L^N_M (t,\omega)$ and $L^D_M(t,\omega)$ for Poissonian potentials}

In this subsection we  restrict our attention to Poissonian
potentials, i.e.
\begin{align}
\label{eq:poiss} V(x,\omega)= \int_{\cG} W(x,y)
\mu^{\omega}({\rm d}y),
\end{align}
where $\mu^\omega$ is the random counting measure corresponding to
the Poisson point process on $\cG$ (defined on the
probability space $(\Omega,\mathcal M, \mathbb{Q})$), and $W$ is a
nonnegative profile function satisfying the basic condition \textbf{(W1)}.

Although we  know that the sequences $ \mathbb E_\qpr L_M^N(t,\omega)$ and $\mathbb E_\qpr L_M^N(t,\omega)$
share their limit points, we do not know for now that either of them is convergent.
To prove the desired convergence we introduce two auxiliary objects
$L^{N^{*}}_M (t,\omega)$, $L^{D^{*}}_M (t,\omega),$ prove that they have the same limit points as $L^{N}_M (t,\omega)$, $L^{D}_M (t,\omega)$, and finally that $L^{N^{*}}_M (t,\omega)$ converges to  a finite limit when $M\to\infty.$

\smallskip

We will need the following `periodization' of $V.$
 \begin{definition} \label{def:Wstar} The family of random fields $(V_M^{*})_{M
\in \Z}$ on $\cG$ given by
$$
V_M^{*}(x,\omega):= \int_{\cG_M} \sum_{y{'} \in \pi_M^{-1}(y)}
W(x,y{'}) \mu^{\omega}({\rm d}y), \quad M \in \Z,
$$ is called
the $M$-periodization   of $V$ \textit{in the Sznitman sense}.
\end{definition}
{The above definition strongly depends on the geometry of $\cG$. In fact, our periodization of the potential function $V$ is a gasket counterpart of that considered in \cite[page 202]{bib:Szn-book} in the Euclidean case. More recently, the periodization of the Poisson random measure was also used in \cite{bib:kpp-sausage} in proving the annealed asymptotics for the Wiener sausage on simple nested fractals.}
By exactly the same argument as in Proposition \ref{prop:Wkato},
one can  check that under the condition \textbf{(W1)}, $V_M^{*}(\cdot,\omega) \in \cK_{\loc}^X$, for every $M \in \Z$, for $\qpr$-almost all $\omega \in \Omega$. Recall that by $V_M$ we have  denoted the usual periodization of $V$,
i.e. $V_M(x,\omega)=V(\pi_M(x),\omega)$.


\smallskip

For $t>0$ and $M \in \Z$ we define:
$$
L^{D^{*}}_M (t,\omega) = \frac{1}{m(\cG_M)} \int_{\cG_M} p(t,x,x) \ex^t_{x,x}\left[e_{V_M^{*}}(t); t < \tau_{\cG_M} \right] m({\rm d}x)
$$
and
$$
L^{N^{*}}_M (t,\omega) = \frac{1}{m(\cG_M)} \int_{\cG_M} p^M(t,x,x)
\ex^{M,t}_{x,x}\left[e_{V_M^{*}}(t)\right] m({\rm d}x).
$$

\

\noindent Repeating the estimates from the proof of Proposition \ref{prop:dir_neum} for
$L^{D^{*}}_M (t,\omega)$ and $L^{N^{*}}_M (t,\omega)$ we get

\begin{corollary}
\label{cor:dir_neum2}
Let the profile $W$ satisfy \textbf{(W1)}. Then for every $t>0$ we have
$$
\mathbb{E}_{\qpr} \left(L^{N^{*}}_M (t,\omega) - L^{D^{*}}_M (t,\omega)\right) =  o(1) \quad \text{as} \quad M \to \infty.
$$
\end{corollary}

Another auxiliary lemma relates the limits of
$\mathbb{E}_{\qpr}L^D_M (t,\omega)$ and
$\mathbb{E}_{\qpr}L^{D^*}_M (t,\omega)$  for Poissonian potentials $V$ with nonnegative
profiles $W$.

\begin{lemma}
\label{lm:dir2}
Let the profile $W$ satisfy conditions \textbf{(W1)}-\textbf{(W2)}. Then for every $t>0$ we have
$$
\mathbb{E}_{\qpr} \left(L^D_M (t,\omega) - L^{D^{*}}_M (t,\omega) \right) =  o(1) \quad \text{as} \quad M \to \infty.
$$
\end{lemma}

\begin{proof}
By \eqref{eq:underQ} for $V$ and a version of this equality for $V_M^*$, we can write
\begin{align*}
\left|\mathbb{E}_{\qpr} (L^D_M (t,\omega) -  L^{D^{*}}_M (t,\omega))\right|  \leq \frac{1}{m(\cG_M)}
\int_{\cG_M} p(t,x,x) \ex^t_{x,x}\left[\left|F_M(w,t)\right| ; t < \tau_{\cG_M} \right] m({\rm d}x),
\end{align*}
where
$$
F_M(w,t) = {\rm e}^{-\nu \int_{\cG}\left(1-{\rm e}^{-\int_0^t W(X_s(w),y) {\rm d}s}\right) m({\rm d}y)}  - {\rm e}^{-\nu \int_{\cG_{M}}\left(1-{\rm e}^{-\int_0^t \sum_{y^{'} \in \pi^{-1}_M(y)} W(X_s(w),y{'}){\rm d}s}\right)m({\rm d}y)}.
$$
Since $W \geq 0$, we get
\begin{align*}
\left|\mathbb{E}_{\qpr}(L^D_M (t,\omega) -  L^{D^{*}}_M (t,\omega))\right| &
\leq  \frac{1}{m(\cG_M)} \int_{\cG_M} p(t,x,x) \ex^t_{x,x}\left[|F_M(w,t)| ; t < \tau_{\cG_M} \right] m({\rm d}x) \\
& \leq  \frac{c(t) m(\cG_M \backslash B(0,2^M-2^{M/2}))}{m(\cG_M)} \\
& \ \ \ + \frac{1}{m(\cG_M)}\int_{B(0,2^M-2^{M/2})} p(t,x,x)
\ex^t_{x,x}\left[ |F_M(w,t)| ; t < \tau_{\cG_M} \right]
m({\rm d}x).
\end{align*}
As before, the first term is bounded by {$c(t) 2^{-\frac{M}{2}\,(d_f-1)}$} and goes to zero as $M \to \infty$. It is
enough to show that the second one goes to 0 as well. Denote it by
$I_M$. We have
\begin{align*}
I_M & \leq \frac{1}{m(\cG_M)}\int_{B(0,2^M-2^{M/2})}
p(t,x,x) \ex^t_{x,x}\left[|F_M(w,t)| ; t <
\tau_{B(0,2^M-2^{M/2})} \right] m({\rm d}x) \\ & \ \ \ \ \ \ \  +
\frac{2}{m(\cG_M)}\int_{B(0,2^M-2^{M/2})} p(t,x,x)
\pr^t_{x,x}\left[ t \geq \tau_{B(0,2^M-2^{M/2})} \right] m({\rm d}x)
=: I_{1,M}+I_{2,M}.
\end{align*}
By an argument similar to that in the proof of Proposition \ref{prop:dir_neum}
 we show that $I_{2,M} \to 0$ as $M \to \infty$. It is enough to estimate $I_{1,M}$.

By the inequality $|{\rm e}^{-x}-{\rm e}^{-y}| \leq |x-y|$, $x,y \geq 0$, the fact that $W \geq 0,$  Fubini, and properties of the measure $m$ we have
\begin{align*}
|F_M(w,t)| & \leq \nu \int_{\cG_M^c}\left(1-{\rm e}^{-\int_0^t W(X_s(w),y) {\rm d}s}\right) m({\rm d}y) \\ & \ \ \ + \nu \int_{\cG_M}\left({\rm e}^{-\int_0^t W(X_s(w),y) {\rm d}s}-{\rm e}^{-\int_0^t \sum_{y^{'} \in \pi^{-1}_M(y)} W(X_s(w),y{'}) {\rm d}s}\right) m({\rm d}y) \\
& \leq \nu \int_0^t \int_{\cG_M^c} W(X_s(w),y) m({\rm d}y) + \nu \int_0^t \int_{\cG_M} \sum_{y \neq y^{'} \in \pi^{-1}_M(y)} W(X_s(w),y{'}) m({\rm d}y) \\
& \leq 2 \nu \int_0^t \int_{\cG_M^c} W(X_s(w),y) m({\rm d}y).
\end{align*}
If the process remains inside $B(0,2^M-2^{M/2})$ up to time $t$ and $y\in\cG_M^c(=B(0,2^M)^c)$, then $d(X_s(w),y)\geq 2^{M/2}$ for all $s\in[0,t].$ It follows
\begin{align*}
I_{1,M} & \leq \frac{2\nu}{m(\cG_M)}\int_{B(0,2^M-2^{M/2})}
p(t,x,x) \ex^t_{x,x}\left[\int_0^t \int_{\cG_M^c} W(X_s(w),y) m({\rm d}y); t <
\tau_{B(0,2^M-2^{M/2})} \right] m({\rm d}x) \\ & \leq c(t,\nu) \sup_{z \in \cG} \int_{B(z,2^{M/2})^c} W(z,y) m({\rm d}y).
\end{align*}
By \textbf{(W2)}, this completes the proof.
\end{proof}

We now know that under conditions \textbf{(W1)}-\textbf{(W2)}, all four sequences:
$\mathbb{E}_\qpr[L^D_M(t,\omega)],$
$\mathbb{E}_\qpr[L^{D^*}_M(t,\omega)],$
$\mathbb{E}_\qpr[L^N_M(t,\omega)],$
$\mathbb{E}_\qpr[L^{N^*}_M(t,\omega)]$ have the same limit points.
We will prove that when additionally \textbf{(W3)} holds, then they are in fact
convergent.

\smallskip

Our main result of this part is the following.

\begin{theorem}\label{th:meanlimit}
Let $V$ be a Poissonian random field with the profile $W$ and let the conditions \textbf{(W1)}-\textbf{(W3)} hold. Then
for every $t>0$, $\mathbb{E}_\qpr[L^D_M(t,\omega)],$
$\mathbb{E}_\qpr[L^{D^*}_M(t,\omega)],$
$\mathbb{E}_\qpr[L^N_M(t,\omega)],$
$\mathbb{E}_\qpr[L^{N^*}_M(t,\omega)]$ are convergent as $M \to
\infty$ to a common finite limit $L(t)$.
\end{theorem}

\begin{proof}
In light of Proposition \ref{prop:dir_neum}, Corollary \ref{cor:dir_neum2} and Lemma \ref{lm:dir2},
it is enough to show that $\mathbb{E}_{\qpr} L^{N^{*}}_M (t,\omega)$ converges to a finite limit
 $L(t)$ as $M \to \infty$. We will prove that $\mathbb{E}_{\qpr} L^{N^{*}}_M (t,\omega)$ is nonincreasing in $M \in \N$.
 Since it is also nonnegative, this will give our assertion.

 First recall that by $\cG_M^{(i)}$, $i=1,2,3$ we have denoted the
isometric copies of $\cG_M$ under $\pi_M^{-1}$ such that
$m\left(\cG_M^{(i)} \cap \cG_M^{(j)}\right)=0$, $i \neq j$, and
$\cG_{M+1} = \bigcup_{i=1}^3 {\cG_M^{(i)}}$. Also, we denoted  by
$\pi_{M,i}$  the restrictions of $\pi_{M}$ to $\cG^{(i)}_M$.

Observe that once the path of the process $X_t$ is fixed, the monotonicity
\begin{align}
\label{eq:monot}
\mathbb{E}_{\qpr} e_{(V^{*}_{M+1})_{M+1}}(t) \leq \mathbb{E}_{\qpr}e_{(V^{*}_{M})_M}(t), \quad t>0,
\end{align}
holds. Indeed, by Definition \ref{def:Wstar}, the exponential formula (\ref{eq:underQ}) applied to
$$f(y)=\1_{\cG_{M+1}}(y) \cdot \int_0^t \sum_{y{'} \in \pi_{M+1}^{-1}(\pi_{M+1}(y))} W(\pi_{M+1}(X_s),y{'})
{\rm d}s$$
and the
standard
inequality $1-{\rm e}^{-\sum_i a_i} \leq \sum_i (1-{\rm e}^{- a_i})$, $a_i
\geq 0$, we have
\begin{align*}
\mathbb{E}_{\qpr} e_{(V^{*}_{M+1})_{M+1}}(t) & = \exp\left(-\nu
\int_{\cG_{M+1}}\left(1-{\rm e}^{-\int_0^t \sum_{y{'} \in
\pi_{M+1}^{-1}(y)} W(\pi_{M+1}(X_s),y{'}) {\rm d}s}\right)m({\rm d}y)\right) \\ & =
\exp\left(-\nu \int_{\cG_M}\sum_{i=1}^3 \left(1-{\rm e}^{-\int_0^t
\sum_{y{'} \in \pi_{M+1}^{-1}(\pi^{-1}_{M,i}(y))} W(\pi_{M+1}(X_s),y{'})
{\rm d}s}\right)m({\rm d}y)\right)
\\ & \leq \exp\left(-\nu \int_{\cG_M}\left(1-{\rm e}^{-\int_0^t \sum_{i=1}^3 \sum_{y{'} \in
\pi_{M+1}^{-1}(\pi^{-1}_{M,i}(y))} W(\pi_{M+1}(X_s),y{'}) {\rm d}s}\right)m({\rm d}y)\right).
\end{align*}
Since for every $M \in \Z$ and $y \in \cG_M \backslash \cV_M$ one has  $\bigcup_{i=1}^3 \pi_{M+1}^{-1}(\pi^{-1}_{M,i}(y)) = \pi_M^{-1}(y)$ and the sets $\pi_{M+1}^{-1}(\pi^{-1}_{M,i}(y))$, $i=1,2,3$, are pairwise disjoint, we get that the last member on the right hand side of the above inequality is equal to
\begin{align} \label{expr:expr}
\exp\left(-\nu \int_{\cG_M}\left(1-{\rm e}^{-\int_0^t \sum_{y{'} \in \pi_M^{-1}(y)} W(\pi_{M+1}(X_s),y{'}) {\rm d}s}\right)m({\rm d}y)\right).
\end{align}
Finally, by \textbf{(W3)} we have
\[\int_0^t\sum_{y'\in\pi_M^{-1}(\pi_M(y))}W(\pi_{M+1}(X_s), y'){\rm d} s \geq \int_0^t \sum_{y'\in\pi_M^{-1}(\pi_M(y))} W(\pi_M(X_s),y')\,{\rm d} s,\quad y\in\cG,\]
and, consequently, the expression in \eqref{expr:expr} is not bigger than
\begin{align*}
\exp \left(-\nu \int_{\cG_M}\left(1-{\rm e}^{-\int_0^t \sum_{y'\in\pi_M^{-1}(y)}W(\pi_M(X_s),y')\,{\rm d} s}\right) m({\rm d} y)\right) = \mathbb{E_Q}e_{(V_M^*)_M}(t),
\end{align*}
 which is exactly (\ref{eq:monot}).

By Lemma \ref{lm:rotation} (a),  the inclusion $\pi^{-1}_{M+1}(x) \subset \pi^{-1}_{M}(\pi_M(x))$ and \eqref{eq:monot}, we have for $x \in \cG_{M+1}$,
\begin{align*}
p^{M+1}(t,x,x) \mathbb{E}_{\qpr} \otimes \ex^{M+1,t}_{x,x}\left[e_{V^{*}_{M+1}}(t)\right] & = \sum_{x^{'} \in \pi^{-1}_{M+1}(x) } p(t,x,x^{\prime})  \ex^t_{x,x^{'}}\left[\mathbb{E}_{\qpr}e_{(V^{*}_{M+1})_{M+1}}(t)\right]  \\
& \leq \sum_{x^{'} \in \pi^{-1}_{M}(\pi_M(x)) } p(t,x,x^{\prime}) \ex^t_{x,x^{'}}\left[\mathbb{E}_{\qpr} e_{(V^{*}_{M+1})_{M+1}}(t)\right] \\
& \leq \sum_{x^{'} \in \pi^{-1}_{M}(\pi_M(x)) } p(t,x,x^{\prime})
\ex^t_{x,x^{'}}\left[\mathbb{E}_{\qpr} e_{(V^{*}_M)_M}(t)\right].
\end{align*}
By the above bound, we get
\begin{align*}
\mathbb{E}_{\qpr} L^{N^*}_{M+1} (t,\omega) & \leq \frac{1}{m(\cG_{M+1})} \int_{\cG_{M+1}} \sum_{x{'} \in \pi^{-1}_{M}(\pi_M(x)) } p(t,x,x^{\prime}) \ex^t_{x,x^{'}}\left[\mathbb{E}_{\qpr} e_{(V^{*}_M)_M}(t)\right] m({\rm d}x) \\
& = \frac{1}{3m(\cG_M)} \sum_{i=1}^3  \int_{\cG_M^{(i)}} \sum_{x^{'} \in \pi^{-1}_{M}(\pi_M(x)) } p(t,x,x^{\prime})  \ex^t_{x,x^{'}}\left[\mathbb{E}_{\qpr} e_{(V^{*}_M)_M}(t)\right] m({\rm d}x) \\
& = \frac{1}{3m(\cG_M)} \sum_{i=1}^3  \int_{\cG_M} \sum_{x^{'} \in
\pi^{-1}_{M}(\pi^{-1}_{M,i}(x)) }
p(t,\pi^{-1}_{M,i}(x),x')  \ex^t_{\pi^{-1}_{M,i}(x),x^{'}}\left[\mathbb{E}_{\qpr} e_{(V^{*}_M)_M}(t)\right] m({\rm d}x)\\
& = \frac{1}{3m(\cG_M)} \sum_{i=1}^3  \int_{\cG_M} \sum_{x^{'} \in \pi^{-1}_{M}(x) } p(t,\pi^{-1}_{M,i}(x),x^{\prime})  \ex^t_{\pi^{-1}_{M,i}(x),x^{'}}\left[\mathbb{E}_{\qpr} e_{(V^{*}_M)_M}(t)\right] m({\rm d} x).
\end{align*}
Now, by Lemma \ref{lm:rotation} (b), we have for
all $x\in \cG_M:$
$$
\sum_{x^{'} \in \pi^{-1}_{M}(x) } p(t,\pi^{-1}_{M,i}(x),x^{\prime})   \ex^t_{\pi^{-1}_{M,i}(x),x^{'}}\left[\mathbb{E}_{\qpr} e_{(V^{*}_M)_M(t)}\right]
= \sum_{x^{'} \in \pi^{-1}_{M}(x) } p(t,x,x^{\prime})  \ex^t_{x,x^{'}}\left[\mathbb{E}_{\qpr} e_{(V^{*}_M)_M(t)}\right]
$$
for every $x \in \cG_M$, i.e. the terms under the last integral
sign do not depend on $i.$ We conclude that
$$
\mathbb{E}_{\qpr} L^{N^*}_{M+1} (t,\omega) \leq \frac{3}{3m(\cG_M)}
\int_{\cG_M} \sum_{x^{'} \in \pi^{-1}_{M}(x) } p(t,x,x^{\prime})
\ex^t_{x,x^{'}}\left[\mathbb{E}_{\qpr} e_{(V^{*}_M)_M(t)}\right] m({\rm d}x)
= \mathbb{E}_{\qpr} L^{N^*}_{M} (t,\omega),
$$
which completes the proof.
\end{proof}

\subsection{The variances lemma}

\begin{lemma}\label{lem:variances-fractal}
Let the profile $W$ satisfy the conditions \textbf{(W1)}-\textbf{(W3)}. Then for any given $t>0$ one has:
\begin{equation}\label{eq:variances-dirichlet}
\sum_{M=1}^\infty \mathbb{E}_\qpr[L_M^D(t,\omega)-\mathbb{E}_\qpr L_M^D(t,\omega)]^2<\infty
\end{equation}
and
\begin{equation}\label{variances-neumann}
\sum_{M=1}^\infty \mathbb{E}_\qpr[L_M^N(t,\omega)-\mathbb{E}_\qpr L_M^N(t,\omega)]^2<\infty.
\end{equation}
\end{lemma}

\begin{proof}
Since for any $L^2-$random variables $\xi,\eta$ one has
 $\mathbf E\xi^2\leq 2\mathbf E \eta^2+2\mathbf{E}(\eta-\xi)^2$, in light of Proposition \ref{prop:dir_neum}, it is enough to prove (\ref{eq:variances-dirichlet}).

\smallskip

Fix $t>0$. Let us introduce the family of measures
\begin{align}
\label{def:newmeas0}
\nu_M:= \left(\frac{1}{{m(\cG_M)}}\int_{\cG_M}p(t,x,x)\pr^t_{x,x} m({\rm d}x)\right)^{\otimes 2} \otimes \qpr^{\otimes 3}, \quad M \in \Z,
\end{align}
on the space $\widetilde \Omega = D([0,t],\cG)^2 \times \Omega^3$. Also, let $(a_M)_{M \in \Z}$ and $(c_M)_{M \in \Z}$ be increasing sequences of positive numbers such that $a_M< c_M<2^{M}$, $M \in \Z$. They will be chosen later on.

For $r>0$ set
$$
V_r(x,\omega):= \int_{B(x,r)}W(x,y)\mu^{\omega}({\rm d}y) \quad \text{and} \quad \widetilde V_r(x,\omega):= \int_{B(x,r)^c}W(x,y)\mu^{\omega}({\rm d}y), \quad r>0
$$
 and then denote
$$
F_M(w,\omega):= {\rm e}^{-\int_0^t V_{a_M}(X_s(w),\omega){\rm d}s},\quad  \quad \widetilde F_M(w,\omega):= {\rm e}^{-\int_0^t \widetilde V_{a_M}(X_s(w),\omega){\rm d}s}, \ \quad M \in \Z.
$$
Note that for every $M$ we have $0 \leq F_M(w,\omega) \leq 1$ and $0 \leq \widetilde F_M(w,\omega)\leq 1$.

Observe that using measures $\nu_M$ and functionals $F_M, \widetilde F_M,$  we can rewrite terms of (\ref{eq:variances-dirichlet}) as
\begin{equation}
\begin{split}
\label{eq:var-dir-1}
\mathbb{E}_\qpr[L_M^D-\mathbb{E}_\qpr L_M^D]^2
& = \int_{\widetilde \Omega} \prod_{i=1}^2   \left(F_M(w_i,\omega_0)\widetilde F_M(w_i,\omega_0) - F_M(w_i,\omega_i)\widetilde F_M(w_i,\omega_i)\right)\1_{\left\{t<\tau_{\cG_M}(w_i)\right\}} \\
& \ \ \ \ \ \ \ \ \ \ \cdot {\rm d} \nu_M(w_1, w_2, \omega_0, \omega_1,\omega_2 )\\
& =: \int_{\widetilde{\Omega}} \mathcal X(w_1, w_2, \omega_0, \omega_1,\omega_2 )\,{\rm d} \nu_M(w_1, w_2, \omega_0, \omega_1,\omega_2 ).
\end{split}
\end{equation}
We split the set $\widetilde{\Omega}$ into three parts:
\begin{eqnarray*}
D_0^M&:=& \left\{(w_1,w_2) \in D([0,t],\cG)^2: \ \text{for every $s \in [0,t]$} \ d(X_s(w_1),X_s(w_2)) > 2c_M \right\}\times \Omega^3,\\
D_1^M&:= & \left\{(w_1,w_2) \in D([0,t],\cG)^2: \right.  \left.d(X_0(w_1),X_0(w_2)) > 2c_M  \right. \\ & & \left. \ \ \  \ \ \ \ \ \ \ \ \ \text{and there exists $s \in (0,t]$} \ \text{such that} \ d(X_s(w_1),X_s(w_2)) \leq 2a_M \right\}\times\Omega^3,\\
D_2^M &:=& \widetilde{\Omega}\setminus (D_0^M\cup D_1^M)
\end{eqnarray*}
and integrate over each of these parts separately.

\smallskip

To estimate the integral over $D_0^M$ first note that
\begin{align*}
\prod_{i=1}^2 & \left(F_M(w_i,\omega_0)\widetilde F_M(w_i,\omega_0) - F_M(w_i,\omega_i)\widetilde F_M(w_i,\omega_i)\right) \\
& = \prod_{i=1}^2 \left(F_M(w_i,\omega_0) - F_M(w_i,\omega_i)\right) + \prod_{i=1}^2 \left(F_M(w_i,\omega_0)\widetilde F_M(w_i,\omega_0) - F_M(w_i,\omega_i)\widetilde F_M(w_i,\omega_i)\right) \\
& \ \ \ \ - \prod_{i=1}^2 \left(F_M(w_i,\omega_0) - F_M(w_i,\omega_i)\right) \\
& = \prod_{i=1}^2 \left(F_M(w_i,\omega_0) - F_M(w_i,\omega_i)\right) + F_M(w_1,\omega_0)F_M(w_2,\omega_0)\left(\widetilde F_M(w_1,\omega_0) \widetilde F_M(w_2,\omega_0)-1\right) \\
& \ \ \ \ \ \ \ \ \ \ \ \ \ \ \ \ \ \ \ \ \ \ \ \ \ \ \ \ \ \ \ \ \ \ \ \ \ \ \ \ \ \ \  + F_M(w_1,\omega_0)F_M(w_2,\omega_2)\left(1-\widetilde F_M(w_1,\omega_0) \widetilde F_M(w_2,\omega_2)\right) \\
& \ \ \ \ \ \ \ \ \ \ \ \ \ \ \ \ \ \ \ \ \ \ \ \ \ \ \ \ \ \ \ \ \ \ \ \ \ \ \ \ \ \ \  + F_M(w_2,\omega_0)F_M(w_1,\omega_1)\left(1-\widetilde F_M(w_2,\omega_0) \widetilde F_M(w_1,\omega_1)\right) \\
& \ \ \ \ \ \ \ \ \ \ \ \ \ \ \ \ \ \ \ \ \ \ \ \ \ \ \ \ \ \ \ \ \ \ \ \ \ \ \ \ \ \ \  + F_M(w_1,\omega_1)F_M(w_2,\omega_2)\left(\widetilde F_M(w_1,\omega_1) \widetilde F_M(w_2,\omega_2)-1\right),
\end{align*}
and, since $|F_M(w_i,\omega_k)|\leq 1$, $i=1,2$, $k=0,1,2$, consequently,
\begin{align*}
\prod_{i=1}^2 & \left(F_M(w_i,\omega_0)\widetilde F_M(w_i,\omega_0) - F_M(w_i,\omega_i)\widetilde F_M(w_i,\omega_i)\right) \\
& \leq \prod_{i=1}^2 \left(F_M(w_i,\omega_0) - F_M(w_i,\omega_i)\right)+ 2- \left(\widetilde F_M(w_1,\omega_0) \widetilde F_M(w_2,\omega_2) +\widetilde F_M(w_2,\omega_0) \widetilde F_M(w_1,\omega_1)\right).
\end{align*}
For a given $M \in \Z$ and a trajectory $X_s(w),$ the functional $F_M(\cdot, w)$ depends only on
those Poisson points that fell onto the set $X_{[0,s]}^{a_M}(w):= \bigcup_{0\leq s\leq t}(X_s(w)+B(0,a_M)).$
Since on the set $D_0^M$ one has $X_{[0,s]}^{a_M}(w_1)\cap X_{[0,s]}^{a_M}(w_2)=\emptyset,$
the random variables $(F_M(w_1,\omega_0)-F_M(w_1,\omega_1))$ and $(F_M(w_2,\omega_0)-F_M(w_2,\omega_2))$ are $\qpr^{\otimes 3}$-independent, and consequently,
$$
\int_{\widetilde \Omega} \prod_{i=1}^2   \left(F_M(w_i,\omega_0) - F_M(w_i,\omega_i)\right)\1_{\left\{t<\tau_{\cG_M}(w_i)\right\}} d \nu_M(\omega_0, \omega_1,\omega_2, w_1, w_2) = 0.
$$
Therefore we have
\begin{align*}
\int_{D_0^M}\mathcal X d\nu_M
& \leq 2 \left( \left(\frac{1}{m(\cG_M)}\int_{\cG_M}p(t,x,x)\pr^t_{x,x}\left[t<\tau_{\cG_M}(w)\right] m({\rm d}x)\right)^2 \right. \\ & \left. \ \ \ \ \ \ \ \ \ \ \ \ \ \ \ \  -\left(\frac{1}{m(\cG_M)}\int_{\cG_M}p(t,x,x)\ex^t_{x,x}\left[ \1_{\left\{t<\tau_{\cG_M}(w)\right\}} \mathbb{E}_{\qpr}\left[\widetilde F_M(w,\omega)\right]\right] m({\rm d}x)\right)^2\right) \\
& \leq \frac{2 {(c_3c_0(t))^2}}{m(\cG_M)}\int_{\cG_M}\ex^t_{x,x}\left[\1_{\left\{t<\tau_{\cG_M}(w)\right\}} \left(1 - \mathbb{E}_{\qpr}\left[\widetilde F_M(w,\omega)\right]\right)\right] m({\rm d}x).
\end{align*}
(The last bound is a consequence of the inequality $a^2-b^2 \leq 2a(a-b)$, where $0 \leq b \leq a$.)
By Jensen and Fubini, since $1-{\rm e}^{-x}\leq x,$ $x \geq 0,$ we get that for every $w$
\begin{align*}
1 - \mathbb{E}_{\qpr}\left[\widetilde F_M(w,\omega)\right] &\leq 1 - {\rm e}^{-\int_0^t \mathbb{E}_{\qpr}\left[\widetilde V_{a_M}(X_s(w),\omega)\right]{\rm d}s} \nonumber \\ &= 1 - {\rm e}^{-\nu\int_0^t  \int_{B(X_s,a_M)^c} W(X_s(w),y) m({\rm d}y){\rm d}s} \leq \nu\int_0^t  \int_{B(X_s,a_M)^c} W(X_s(w),y) m({\rm d}y){\rm d}s\\
&\leq \nu t \sup_{z\in\cG}\int_{B(z,a_M)^c} W(z,y) \, m ({\rm d} y).
\end{align*}
{Thus
\begin{align}\label{eq:var-dir-2}
\int_{D_0^M}\mathcal X d\nu_M \leq c(t,\nu) \sup_{x \in \cG} \int_{B(x,a_M)^c} W(x,y) m({\rm d}y).
\end{align}}

 \smallskip

 On the set $D_1^M,$ one necessarily has
 $$
\sup_{s \in (0,t]}d(X_0(w_i),X_s(w_i)) > c_M - a_M, \quad \ \text{for} \ i=1 \ \text{or} \ i=2,
$$
 and since $|F_M(w_i,\omega_k)|\leq 1$, $i=1,2$, $k=0,1,2$, the integral over $D_1^M$ is not bigger than
 \begin{equation}\label{eq:int-d1}
 \frac{4}{m(\cG_M)}\int_{\cG_M} p(t,x,x) \mathbf P^t_{x,x}[\sup_{s\in(0,t]}d(X_0,X_s)>c_M-a_M] \,m({\rm d}x).
 \end{equation}
 By the bridge symmetry we obtain that for every $x\in \cG_M$
 \begin{eqnarray}\label{eq:int-d1-1}
 p(t,x,x)\mathbf E^t_{x,x}[\sup_{s\in(0,t]}d(X_0,X_s)>c_M-a_M] &\leq&  2p(t,x,x) \mathbf P^t_{x,x}[ \sup_{s\in(0,t/2]}d(X_0,X_s)>c_M-a_M]\nonumber\\
 &=& 2\mathbf{E}_x\left[\mathbf{1}\{\sup_{s\in(0,t/2]}d(X_0,X_s)>c_M-a_M\} p(t/2, X_{t/2},x)\right]\nonumber\\
 &\leq & 2c_3c_0(t) \sup_{x\in \cG_M} \mathbf{P}_x\left[\sup_{s\in(0,t/2]} d(X_s,X_0)>c_M-a_M\right].
  \end{eqnarray}

 \smallskip

Finally, to estimate the integral over $D_2^M$ we just estimate the measure of this set
(the integrand is not bigger than 2):
\begin{equation}\label{eq:int-d2}\nu_M(D_2^M)\leq  \, \frac{c_3c_0(t) m(\{(x,y)\in \cG_M\times \cG_M: d(x,y)\leq 2c_M\})}{m(\cG_M)^2}\leq \frac{c(t) (2c_M)^{d_f}}{3^M}.
\end{equation}
If we choose $c_M=2^{M/2}$ and $a_M=2^{M/4}$, we get that \eqref{eq:var-dir-2} and (\ref{eq:int-d1-1}) are terms of convergent series, due to \textbf{(W2)} and Lemma \ref{lm:supr_sum}, respectively. The estimate in (\ref{eq:int-d2}) is summable as well.
The Lemma follows.
\end{proof}

\subsection{Convergence-conclusion}
\noindent Having proven Theorem \ref{th:meanlimit} and Lemma \ref{lem:variances-fractal}, we can give the proof of the main result.

\begin{theorem}\label{th:main-fractal}
Let the profile $W$ satisfies the conditions \textbf{(W1)}-\textbf{(W3)}. Then the random measures $l_M^D(\omega)$
and $l_M^N(\omega)$ are $\qpr-$almost surely vaguely convergent to a common nonrandom limit measure $l$ on $\R_+$.
\end{theorem}

\begin{proof} We prove the statement for the measures $l_M^D(\omega),$ the proof for $l_M^N(\omega)$ is identical.
A classical Borel-Cantelli lemma argument gives that for any fixed $t>0,$ the Laplace transforms $L_M^D(t,\omega)$ converge $\qpr-$almost surely to the limit $L(t).$
Therefore, $\qpr-$a.s., the same statement holds for all rational $t$'s.

Consequently, $\mathbb Q-$a.s., sequences $L_M^D(t,\omega)$ converge to $L(t)$ for all $t>0.$ In particular, $L_M^D(1,\omega)$ is convergent to $L(1),$ and so the measures
$\widetilde{l}_M^D(\omega)({\rm d}\lambda)= {\rm e}^{-\lambda}\,l_M^D(\omega)({\rm d}\lambda)$ are finite. Since any sequence of finite measures on $\mathbb R_+$ is vaguely relatively compact (see \cite[Lemma A9]{bib:SSV}), the measures $\widetilde{l}_M^D(\omega),$ and consequently also $l_M^D(\omega)$, are vaguely convergent to the measure having $L$  for its Laplace transform.
\end{proof}

\subsection{{Existence of IDS for subordinate Brownian motions killed by Poissonian obstacles}}
{In this subsection we discuss the problem of existence of IDS for subordinate Brownian motions on the Sierpi\'nski gasket which are killed upon coming to the set of random obstacles $\cO(\omega):=\bigcup_i \overline B(y_i(\omega),a)$, where $\left\{y_i(\omega)\right\}_i$ is a realization of the Poisson point process over the probability space $(\Omega, \cM, \qpr)$ and $a > 0$ is the radius of the obstacles.  Informally speaking, such a system may be seen as the motion of a particle in the random environment given by the potential of the form
$$
V(x,\omega) =\sum_{y_i(\omega)} W(x,y_i(\omega)), \ \quad \ \text{where} \ \quad \ W(x,y)= \infty \cdot \1_{B(x,a)}(y).
$$
Formally, in this case, we are interested in the spectral problem for the semigroups
\begin{align}
\label{def:sem-dir-obst}
P_t^{D,M,\omega} f(x) = \ex_x \left[f(X_t); t< T^{X}_{\cO(\omega)}, t<\tau_{\cG_M}\right], \quad f \in L^2(\cG_M,m), \quad M \in \Z, \quad t>0,
\end{align}
and
\begin{align}
\label{def:sem-neu-obst}
P_t^{N,M,\omega} f(x) = \ex^M_x \left[f(X^M_t); t< T^{X^M}_{\cO(\omega)}\right],  \quad f \in L^2(\cG_M,m), \quad M \in \Z, \quad t>0,
\end{align}
where $T^{X}_{\cO(\omega)}$ and $T^{X^M}_{\cO(\omega)}$ are, respectively, the first hitting times of the set $\cO(\omega)$ for the subordinate process $X$ and its 'reflected' counterpart $X^M$. }

{The existence of IDS for such a problem is not directly covered by our results, but it can be proved by a modification of the argument presented in this paper for Poissonian potentials. To this goal, the Feynman-Kac functional $e^{-\int_0^t V( \bullet ,\omega) ds}$ should be replaced by the condition $\left\{t< T^{\bullet}_{\cO(\omega)}\right\}$ for an appropriate process $X$ or $X^M$. Instead of 'periodization' of $V$ one should consider the 'periodization' of the Poisson random measure. By this simplification, all results of Section 3 can be immediately rearranged to the case of killing obstacles.

More precisely, for given $M\geq 0,$ we consider two possible `periodizations' of the obstacle set:
\begin{align}\label{eq:obs-per-usu}
\mathcal O_M(\omega)= & \pi_M^{-1}(\mathcal O(\omega)\cap \cG_M),\\
\label{eq:obs-per-szn}
\mathcal O_M^*(\omega)=& \bigcup_{y\in\mathcal N_M(\omega)} \overline{B}(y,a),
\end{align}
where
\[\mathcal N_M(\omega)= \pi_M^{-1}((y_i(\omega))_i\cap \cG_M).\]
The difference between those sets is that $\mathcal O_M(\omega)$ arises as
a usual periodization of the part of the set $\mathcal O(\omega)$ that
lies within $\cG_M,$ whereas
to obtain $\mathcal O^*_M(\omega),$ one periodizes just the Poisson
points that fell into $\cG_M,$ and then builds obstacles on this periodic set.

As in the potential case, we consider the Laplace transforms of the  empirical measures based on the spectra of the semigroups
$(P_t^{D,M,\omega})$ and $(P_t^{N,M,\omega}),$
namely (we keep the same notation),
\begin{equation}\label{eq:lapl-pois-D}
L_M^{D}(t,\omega)=\frac{1}{m(\cG_M)}\int_{\cG_M}
p(t,x,x) \mathbf P^t_{x,x}\left[T^X_{\mathcal O(\omega)}>t, \tau_{\cG_M}>t\right]\,{\rm d}m(x)
\end{equation}
\begin{equation}\label{eq:lapl-pois-N}
L_M^{N}(t,\omega)=\frac{1}{m(\cG_M)}\int_{\cG_M}
p^M(t,x,x) \mathbf P^{M,t}_{x,x}\left[T^{X^M}_{\mathcal O(\omega)}>t\right]\,{\rm d}m(x).
\end{equation}
Observe that in both these expressions we may replace the set $\mathcal O(\omega)$ with $\mathcal O_M(\omega).$  Expressions  $L_M^{D^*}(t,\omega)$ and $L_M^{N^*}(t,\omega)$ arise similarly, but with $\mathcal O_M^*(\omega)$ replacing $\mathcal O_M(\omega).$

As above,
we prove the following.
\begin{proposition}\label{prop:Pois-mean} Let $t>0$ be given.
For the quantities $L^D_M(t,\omega),$  $L^N_M(t,\omega),$ $L^{D*}_M(t,\omega),$ $L^{N^*}_M(t,\omega),$ defined above, we have:
\begin{enumerate}
\item[(1)] \(\sum_{M=1}^\infty \mathbb{E_Q}[L_M^N(t,\omega)-L_M^D(t,\omega)]^2<\infty,\)
    \item[(2)] \(\mathbb{E_Q}[L_M^{N^*}(t,\omega) -L_M^{D^*}(t,\omega)] =o(1),\) as $M\to\infty,$
        \item[(3)] \(\mathbb{E_Q}[L_M^{D}(t,\omega) -L_M^{D^*}(t,\omega)] =o(1),\) as $M\to\infty,$
            \item[(4)] \(\mathbb{E_Q}[L^{N^*}_M(t,\omega)]\) is convergent to a finite limit when $M\to\infty,$
                \item[(5)] consequently, $\mathbb{E_Q}[L^D_M(t,\omega)],$ $\mathbb{E_Q}[L^N_M(t,\omega)],$ $\mathbb{E_Q}[L^{D^*}_M(t,\omega)]$ are convergent as well.
\end{enumerate}
\end{proposition}

\begin{proof} In all that follows we assume that $M$ is large enough to have $2^M>a,$ where $a>0$ is the radius of obstacles. We first show (1). We have
\begin{align}\nonumber
L^D_M(t,\omega)  = & \frac{1}{m(\cG_M)}\int_{\cG_M} p(t,x,x)\mathbf P_{x,x}^t[T^X_{\mathcal O_M(\omega)}>t]\,{\rm d}m(x)\\ \nonumber
&-\frac{1}{m(\cG_M)}\int_{\cG_M} p(t,x,x)\mathbf P_{x,x}^t[T^X_{\mathcal O_M(\omega)}>t, \tau_{\cG_M}\leq t]\,{\rm d}m(x)
\end{align}
and
\[
L_M^N(t,\omega)= \frac{1}{m(\cG_M)}\int_{\cG_M} \sum_{x'\in\pi_M^{-1}(x)}p(t,x,x') \mathbf P^t_{x,x'}[T^X_{\mathcal O_M(\omega)}>t]\,{\rm d}m(x),
\]
therefore
\[0\leq L_M^N(t,\omega)-L_M^D(t,\omega)\leq R_{1,M}(t)+R_{2,M}(t),\]
where $R_{1,M}(t)$ and $R_{2,M}(t)$ are given by (\ref{eq:r-jeden}) and (\ref{eq:r-dwa}).  Proof of (1) follows then as the proof of Proposition
\ref{prop:dir_neum}.

To get (2), we just repeat the steps leading to (1), with obstacle
set being $\mathcal O_M^*.$

Proof of (3) requires somewhat more work.
 If  we denote $\cG_M^{-a}=\{x\in\cG_M: d(x, \cG_M^c)>a\}, $ then
since $p(t,x,x)\leq c(t)$ and
 $\frac{m(\cG_M\setminus \cG_M^{-a})}{m(\cG_M)} =o(1)$
(by an argument similar to that following (\ref{eq:zn})), we have
\[\mathbb{E_Q}L_M^D(t,\omega)=\frac{1}{m(\cG_M)}\int_{\cG_M^{-a}} p(t,x,x) \mathbb{E_Q}\mathbf E_{x,x}^t[T^X_{\mathcal O_M}>t, \tau_{\cG_M}>t] {\rm d}m(x) + o(1),\]and
\[\mathbb{E_Q}L_M^{D^*}(t,\omega)=\frac{1}{m(\cG_M)}\int_{\cG_M^{-a}} p(t,x,x) \mathbb{E_Q} \mathbf E_{x,x}^t[T^X_{\mathcal O_M^*}>t, \tau_{\cG_M}>t] {\rm d}m(x) + o(1).\]
Next, observe that
\begin{align*}
\mathbb{E_Q}\mathbf{E}_{x,x}^t[T^X_{\mathcal O_M^*}>t,\tau_{\cG_M}>t]= & \mathbb{E_Q} \mathbf{E}_{x,x}^t[T^X_{\mathcal O_M}>t,\tau_{\cG_M^{-a}}>t]\\
&+ \mathbb{E_Q} \mathbf{E}_{x,x}^t[T^X_{\mathcal O_M^*}>t,\tau_{\cG_M}>t, \tau_{\cG_M^{-a}}\leq t].
\end{align*}
It follows
\begin{align*}
\mathbb{E_Q}[L_M^D(t,\omega)-L_M^{D^*}(t,\omega)]=& \frac{1}{m(\cG_M)}
\int_{\cG_M^{-a}} p(t,x,x) \mathbb{E_Q}\mathbf E_{x,x}^t\left[ T^X_{\mathcal O_M}>t, \tau_{\cG_M}>t, \tau_{\cG_M^{-a}}\leq t\right]\,{\rm d}m(x)\\
&- \frac{1}{m(\cG_M)}
\int_{\cG_M^{-a}} p(t,x,x) \mathbb{E_Q}\mathbf E_{x,x}^t\left[ T^X_{\mathcal O_M^*}>t, \tau_{\cG_M}>t, \tau_{\cG_M^{-a}}\leq t\right]\,{\rm d}m(x)\\ &+ o(1).
\end{align*}
Both members above are then estimated by $\frac{1}{m(\cG_M)}\int_{\cG_M^{-a}}
p(t,x,x)\mathbf P_{x,x}^t[\tau_{\cG_M^{-a}}\leq t]\,{\rm d}m(x),$ which goes to zero as $M\to \infty$ (consider separately the integrals over the sets
$\cG_M^{-2a}$ and $\cG_M^{-a}\setminus \cG_M^{-2a},$ then proceed as in the proof of Proposition \ref{prop:dir_neum}).

\smallskip

(4) and (5). Clearly, it is enough to show (4). Again, we prove that
the sequence $\mathbb{E_Q}[L^{N^*}_M(t,\omega)]$ is monotone decreasing.
The key observation, replacing  formula (\ref{eq:monot}), is now
\begin{align}\label{eq:monot-poiss}
 \qpr[ T^{X^{M+1}}_{\cO_{M+1}^*(\omega)}>t]\leq \qpr[T^{X^M}_{\cO_M^*(\omega)}>t],\quad M \in \Z.
 \end{align}
 Indeed, for fixed $M\in\Z$ and a given trajectory of $X^M$ we have
 \[\mathbb Q[T^{X^M(w)}_{\mathcal O_M^*(\omega)}>t] = {\rm e}^{-\nu m(X_{[0,t]}^{M,a}(w))},\] where $X^{M,a}_{[0,t]}(w) := \bigcup_{0 \leq s \leq t} \left(X^M_s(w) + B(0,a)\right)$.
This comes as a consequence of the definition of the obstacle set
$\mathcal O_M^*(\omega):$ the trajectory of $X^M$ does not come to the
obstacle set if and only if no obstacles fall into the $a-$vicinity of the trajectory up to time $t.$
To finish the proof of (\ref{eq:monot-poiss}), use the volume monotonicity:
$m(X^{M+1,a}_{[0,t]}(w)) \geq m(X^{M,a}_{[0,t]}(w)),$ which is a consequence of a general fact
\[m(\pi_{M+1}(A))\geq m(\pi_M(A)),\]
valid for an arbitrary measurable set $A\subset\cG.$

Using (\ref{eq:monot-poiss}), we conclude the proof identically as that of Theorem \ref{th:meanlimit}.
\end{proof}
The other ingredient in the proof of the existence of IDS is the variances lemma (Lem. \ref{lem:variances-fractal}). Its proof in the obstacle case is identical as before, and in fact shorter -- as the range of interaction with obstacles is finite, there is no need to introduce
quantities $\widetilde F_M.$

Therefore we obtain:

\begin{theorem}\label{th:IDS-obstacles} Let $a> 0$ be given.
The random measures $l_M^D(\omega)$
and $l_M^N(\omega)$ given by (\ref{eq:el-d}), (\ref{eq:el-n})  in the case of killing obstacles with radius $a,$  are $\qpr-$almost surely vaguely convergent to a common nonrandom limit measure $l$ on $\R_+$.
\end{theorem}

\begin{remark} {\rm When the process $X$ is point-recurrent, then $a=0$ is permitted as well.}
\end{remark}

\section{Examples of profile functions $W$} \label{sec:ex}

In this section we give and discuss some examples of profile functions $W$ which satisfy all of our
three regularity conditions \textbf{(W1)}-\textbf{(W3)}.

\begin{example} \label{ex:exW1}\rm{
Fix $M_0 \in \Z$ and let the function $\psi:\cG_{M_0} \to [0,\infty)$ be such that $\psi \in L^1(\cG_{M_0},m)$. Define
$$
W(x,y):=
\left\{
\begin{array}{ll}
\psi(\pi_{M_0}(y)), & \text{when} \ x, y \in \Delta_{M_0}(z_0), \ \text{for some} \ z_0 \in \cG \backslash \mathcal V_{M_0}, \\
0,       & \text{otherwise} .
\end{array}
\right.
$$
We see that for each fixed $x \in \cG$, $W(x,\cdot)$ is a function supported in the ball $B(x,2^{M_0})$ and for each fixed $y \in \cG$, $W(\cdot,y)$ is a simple function taking the value $\psi(\pi_{M_0}(y))$ on the triangle (or two adjacent triangles) of size $2^{M_0}$ containing $y$ and $0$ otherwise. Therefore, both conditions \textbf{(W1)} and \textbf{(W2)} are immediately satisfied. For every $x \in \cG$ we denote
$$A_{M_0}(x):=\left\{y{'} \in \cG: \ \text{there exists} \ z_0 \in \cG \backslash \mathcal V_{M_0} \ \text{such that} \ x, y{'} \in \Delta_{M_0}(z_0)\right\}.$$
{One can observe that for every natural $M > M_0$ and any $x, y \in \cG$ the sets $\pi_M^{-1}(\pi_M(y)) \cap A_{M_0}(\pi_M(x))$ and $\pi_M^{-1}(\pi_M(y)) \cap A_{M_0}(\pi_{M+1}(x))$ have the same number of elements (zero, one, or two). In  this case,
$$\pi_{M_0} \left(\pi_M^{-1}(\pi_M(y)) \cap A_{M_0}(\pi_{M+1}(x))\right) = \pi_{M_0} \left(\pi_M^{-1}(\pi_M(y)) \cap A_{M_0}(\pi_{M}(x))\right).$$
Hence
\begin{align*}
\sum_{y{'} \in \pi_M^{-1}(\pi_M(y))} W(\pi_{M+1}(x),y{'}) & = \sum_{y{'} \in \pi_M^{-1}(\pi_M(y))}\psi(\pi_{M_0}(y{'})) \1_{A_{M_0}(\pi_{M+1}(x))}(y{'})\\
& = \sum_{y{'} \in \pi_M^{-1}(\pi_M(y))} \psi(\pi_{M_0}(y{'})) \1_{A_{M_0}(\pi_M(x))}(y{'}) = \sum_{y{'} \in \pi_M^{-1}(\pi_M(y))} W(\pi_M(x),y{'}),
\end{align*}
which gives \textbf{(W3)}.}
}
\end{example}

\begin{example}{\rm
Let $\varphi:[0,\infty) \to [0,\infty)$ be a function satisfying the following conditions.
\begin{itemize}
\item[(1)] There exists $R>0$ such that $\varphi(x)=0$ for all $x \in (R,\infty)$.
\item[(2)] For every $y \in \cG$ one has {Let} $\varphi(d(\cdot,y))  \in \cK_{\loc}^X$.
\end{itemize}
For such a function $\varphi$ we define
\begin{align}\label{ex:phi}
W(x,y):=\varphi(d(x,y)), \quad x, y \in \cG.
\end{align}
We have $W(x,y)=W(y,x)$ for all $x, y \in \cG$, and (1)--(2) immediately imply both conditions \textbf{(W1)} and \textbf{(W2)}. We will now check \textbf{(W3)}. For every natural $M \geq M_0:=\left\lceil \log_2 R \right\rceil$ and every $x, y \in \cG$ we denote $D_M(x,y):= \pi^{-1}_M(\pi_M(y)) \cap B(x,R)$. {For every $M \geq M_0$ and $x, y \in \cG$,  the set $D_M(\pi_{M}(x),y)$ has no more elements than $D_M(\pi_{M+1}(x),y)$. Moreover, for every $z{'} \in D_M(\pi_M(x),y)$ there is exactly one $y{'} \in D_M(\pi_{M+1}(x),y)$ (different for different choices of $z{'}$) such that $d(\pi_{M+1}(x),y{'})=d(\pi_M(x),z{'})$. This gives that, for every $M \geq M_0$ and arbitrary $x, y \in \cG,$ we have
\begin{align*}
\sum_{y{'} \in \pi_M^{-1}(\pi_M(y))} W(\pi_{M+1}(x),y{'}) & = \sum_{y{'} \in D_M(\pi_{M+1}(x),y)} \varphi(d(\pi_{M+1}(x),y{'})) \\ & \geq \sum_{z{'} \in D_M(\pi_M(x),y)} \varphi(d(\pi_M(x),z{'})) = \sum_{y{'} \in \pi_M^{-1}(\pi_M(y))} W(\pi_M(x),y{'}),
\end{align*}
which completes the justification of \textbf{(W3)} for the profile $W$ given by \eqref{ex:phi}}.

We note that when $\varphi$ is bounded, then the condition (2) is automatically satisfied for any subordinate
Brownian motion with subordinator satisfying our Assumption \ref{ass:ass1}. Singular functions $\varphi$ are also allowed, but the possible type of singularity strongly depends on the subordinator. {For example, let $\phi(\lambda) = \lambda^{\alpha/d_w
}$ with $\alpha \in (0,d_w)$ and $\varphi(s)=s^{-\gamma} \1_{s \leq R}$ with $\gamma>0$ and some $R>0$. If $\alpha \in (d_f,d_w)$, then $\int_0^1 c_0(t) dt < \infty$ and (2) is satisfied whenever $\gamma < d_f$. When $\alpha \in (0,d_f]$, then $\int_0^1 c_0(t) dt = \infty$ and (2) is satisfied if $\gamma < \alpha$. The latter assertion can be established by extending a very general argument in \cite[Lemma 7 and estimates below it]{bib:BoSz} to the case of this specific process on the Sierpi\'nski gasket.}
}
\end{example}

We now give an example of profile $W$ with unbounded support, which naturally extends Example \ref{ex:exW1}.
\begin{example}{\rm
Let $(a_n)_{n \in \Z}$ be a sequence of nonnegative numbers such that
\begin{align} \label{eq:intexp}
\sum_{M=1}^{\infty} \sum_{n=[M/4]+1}^{\infty} 3^n a_n < \infty.
\end{align}
Define
\begin{align*}
W(x,y) :=\left\{
\begin{array}{ll}
a_0 & \text{when} \quad x, y \in \cG \backslash \mathcal V_0, \ y \in \Delta_{0}(x), \\
a_n & \text{when} \quad x, y \in \cG \backslash \mathcal V_0, \ y \in \Delta_{n}(x) \backslash \Delta_{n-1}(x), \ n = 1, 2, 3, ..., \\
0 & \text{otherwise}.
\end{array}
\right.
\end{align*}
It is easy to see that  \textbf{(W1)} is satisfied. By the inequality
$$
\int_{B(x,2^{[M/4]})^c} W(x,y) m({\rm d}y) \leq \int_{\Delta_{[M/4]}^c(x)} W(x,y) m({\rm d}y) = \sum_{n=[M/4]+1}^{\infty} 2 \cdot 3^{n-1} a_n, \quad M \in \Z, \quad x \in \cG \backslash \mathcal V_0,
$$
and \eqref{eq:intexp}, we get that also \textbf{(W2)} holds. It is enough to justify \textbf{(W3)}. Fix $M \in \Z$ and $x, y \in \cG \backslash \mathcal V_0$. Recall that $\cG_{M+1} = \bigcup_{i=1}^3 \cG^{(i)}_M$. Clearly, each of the three sets $\cG^{(i)}_M \cap \pi_M^{-1}(\pi_M(y))$, $i=1,2,3$, has exactly one element. If $\pi_{M+1}(x) \in \cG^{(1)}_M=\cG_M$, then $\pi_{M+1}(x) = \pi_{M}(x)$ and \textbf{(W3)} follows directly. With no loss of generality we may assume that $\pi_{M+1}(x) \in \cG^{(3)}_M$ (the case of $\pi_{M+1}(x) \in \cG^{(2)}_M$ is analogous). We will need the following observations.
\begin{itemize}
\item[(1)] Since $\Delta_{M+1}(\pi_M(x)) = \Delta_{M+1}(\pi_{M+1}(x)) = \cG_{M+1}$, we have
$$
\sum_{y{'} \in \pi_M^{-1}(\pi_M(y))} W(\pi_{M+1}(x),y{'}) \1_{\cG^c_{M+1}}(y{'}) = \sum_{y{'} \in \pi_M^{-1}(\pi_M(y))} W(\pi_{M}(x),y{'}) \1_{\cG^c_{M+1}}(y{'}).
$$
\item[(2)] We have $\Delta_M(\pi_{M}(x)) = \cG^{(1)}_M$, $\Delta_M(\pi_{M+1}(x)) = \cG^{(3)}_M$ and $\pi_M(\pi_{M+1}(x))=\pi_M(x)$. For every $n \in \left\{1,2,...,M\right\}$, the only element of $\pi_M^{-1}(\pi_M(y)) \cap \cG^{(3)}_M$ belongs to $\Delta_n(\pi_{M+1}(x)) \backslash \Delta_{n-1}(\pi_{M+1}(x))$ if and only if the only element of $\pi_M^{-1}(\pi_M(y)) \cap \cG^{(1)}_M$ belongs to $\Delta_n(\pi_{M}(x)) \backslash \Delta_{n-1}(\pi_{M}(x))$. Therefore,
$$
W(\pi_{M+1}(x),y{'}) \1_{\pi_M^{-1}(\pi_M(y)) \cap \cG^{(3)}_{M}}(y{'}) = W(\pi_{M}(x),y{'}) \1_{\pi_M^{-1}(\pi_M(y)) \cap \cG^{(1)}_{M} }(y{'})
$$
and
$$
\sum_{y{'} \in \pi_M^{-1}(\pi_M(y))} W(\pi_{M+1}(x),y{'}) \1_{\cG^{(1)}_{M} \cup \cG^{(2)}_{M}}(y{'}) = \sum_{y{'} \in \pi_M^{-1}(\pi_M(y))} W(\pi_{M}(x),y{'}) \1_{\cG^{(2)}_{M} \cup \cG^{(3)}_{M}}(y{'}).
$$
\end{itemize}
By these observations, we have
\begin{align*}
\sum_{y{'} \in \pi_M^{-1}(\pi_M(y))} &  W(\pi_{M+1}(x),y{'}) = \sum_{y{'} \in \pi_M^{-1}(\pi_M(y))} W(\pi_{M+1}(x),y{'})\left( \1_{\cG_{M+1}}(y{'}) + \1_{\cG^c_{M+1}}(y{'}) \right) \\
& = \sum_{y{'} \in \pi_M^{-1}(\pi_M(y))} W(\pi_{M+1}(x),y{'})   \1_{\cG^{(1)}_{M} \cup \cG^{(2)}_{M}}(y{'}) + \sum_{y{'} \in \pi_M^{-1}(\pi_M(y))} W(\pi_{M+1}(x),y{'}) \1_{\cG^{(3)}_{M}}(y{'}) \\
& \ \ \ \ \ \ \ \ \  + \sum_{y{'} \in \pi_M^{-1}(\pi_M(y))} W(\pi_{M+1}(x),y{'}) \1_{\cG^c_{M+1}}(y{'})  \\
& = \sum_{y{'} \in \pi_M^{-1}(\pi_M(y))} W(\pi_{M}(x),y{'})  \1_{\cG^{(2)}_{M} \cup \cG^{(3)}_{M}}(y{'}) + \sum_{y{'} \in \pi_M^{-1}(\pi_M(y))} W(\pi_{M}(x),y{'}) \1_{\cG^{(1)}_{M}}(y{'}) \\ & \ \ \ \ \ \ \ \ \  + \sum_{y{'} \in \pi_M^{-1}(\pi_M(y))} W(\pi_{M}(x),y{'}) \1_{\cG^c_{M+1}}(y{'})  \\           & = \sum_{y{'} \in \pi_M^{-1}(\pi_M(y))} W(\pi_{M}(x),y{'}),\end{align*}
which gives \textbf{(W3)}.}
\end{example}

\bigskip
 {
\noindent
\textbf{Acknowledgements.} We would like to thank the anonymous referee for his valuable suggestions and comments.
}

\end{document}